\documentclass[12pt]{amsart}
\usepackage{amsmath}
\usepackage{amssymb}
\usepackage[all]{xy}
\usepackage{longtable}
\usepackage[osf,sc]{mathpazo}
\usepackage{euscript}
\usepackage{calrsfs}
\usepackage{color}
\usepackage{multirow,bigdelim}
\usepackage{gfsartemisia-euler}

\newtheorem{theorem}[equation]{Theorem}
\newtheorem{lemma}[equation]{Lemma}
\newtheorem{proposition}[equation]{Proposition}
\newtheorem{corollary}[equation]{Corollary}
\newtheorem{conjecture}[equation]{Conjecture}
\newtheorem{definition}[equation]{Definition}

\theoremstyle{definition}
\newtheorem{example}[equation]{Example}

\theoremstyle{remark}
\newtheorem{remark}[equation]{Remark}

\numberwithin{equation}{subsection}

\allowdisplaybreaks[1]

\newcommand{\FF}{\mathbb{F}}

\newcommand{\QQ}{\mathbb{Q}}
\newcommand{\RR}{\mathbb{R}}

\newcommand{\TT}{\mathbb{T}}
\newcommand{\GG}{\mathbb{G}}

\newcommand{\CC}{\mathbb{C}}

\newcommand{\NN}{\mathbb{N}}

\newcommand{\Fq}{\mathbb{F}_{q}}

\newcommand{\be}{\mathbf{e}}

\newcommand{\bm}{\mathbf{m}}

\newcommand{\bx}{\mathbf{x}}
\newcommand{\bX}{\mathbf{X}}
\newcommand{\bu}{\mathbf{u}}
\newcommand{\bv}{\mathbf{v}}

\newcommand{\bC}{\mathbf{C}}

\newcommand{\bz}{\mathbf{z}}

\newcommand{\bj}{\mathbf{j}}
\newcommand{\bw}{\mathbf{w}}
\newcommand{\by}{\mathbf{y}}

\newcommand{\cM}{\mathcal{M}}
\newcommand{\cN}{\mathcal{N}}

\newcommand{\cT}{\mathcal{T}}

\newcommand{\cL}{\mathcal{L}}

\newcommand{\cZ}{\mathcal{Z}}

\DeclareMathAlphabet{\matheur}{U}{eur}{m}{n}

\newcommand{\fs}{\mathfrak{s}}

 \DeclareMathOperator{\Lie}{Lie}
\DeclareMathOperator{\Ker}{Ker} \DeclareMathOperator{\GL}{GL}
\DeclareMathOperator{\Mat}{Mat} 
\DeclareMathOperator{\End}{End}

 \DeclareMathOperator{\wt}{wt}
  
\DeclareMathOperator{\Li}{Li}
\DeclareMathOperator{\rank}{rank}

\DeclareMathOperator{\dep}{dep}
\DeclareMathOperator{\Span}{Span}

\newcommand{\ok}{\overline{k}}

\newcommand{\tr}{\mathrm{tr}}

\newcommand{\Ga}{{\GG_{\mathrm{a}}}}

\newcommand{\Lis}{\Li^{\star}}

\newcommand{\laurent}[2]{{#1 (\!( #2 )\!)}}

\begin{document}

\title[On a conjecture of Furusho over function fields]{{\large{On a conjecture of Furusho over function fields}}}

\author{Chieh-Yu Chang}
\address{Department of Mathematics, National Tsing Hua University, Hsinchu City 30042, Taiwan
  R.O.C.}

\email{cychang@math.nthu.edu.tw}

\author{Yoshinori Mishiba }
\address{Department of Mathematical Sciences, University of the Ryukyus, 1 Senbaru, Nishihara-cho, Okinawa 903-0213, Japan}
\email{mishiba@sci.u-ryukyu.ac.jp}

\thanks{The first author was partially supported by MOST Grant
  107-2628-M-007-002-MY4 and Foundation for The Advancement of Outstanding Scholarship.}
\thanks{The second author was supported by JSPS KAKENHI Grant Numbers JP15K17525 and JP18K13398.}


\date{June 24, 2020}

\begin{abstract}
In the classical theory of multiple zeta values (MZV's), Furusho proposed a conjecture asserting that the $p$-adic MZV's satisfy the same $\QQ$-linear relations that their corresponding real-valued MZV counterparts satisfy. In this paper, we verify a stronger version of a function field analogue of Furusho's conjecture in the sense that we are able to deal with all linear relations over an algebraic closure of the given rational function field, not just the rational linear relations.    To each tuple of positive integers $\fs=(s_1, ..., s_r)$, we construct a corresponding $t$-module together with a specific rational point. The fine resolution (via fiber coproduct) of this construction actually allows us to obtain nice logarithmic interpretations for both the $\infty$-adic MZV and $v$-adic MZV at $\fs$, completely generalizing the work of Anderson-Thakur~\cite{AT90} in the case of $r=1$. Furthermore it enables us to apply Yu's sub-t-module theorem~\cite{Yu97}, connecting any $\infty$-adic linear relation on MZV's with a  sub-$t$-module of a corresponding giant $t$-module. This makes it possible to arrive at the same linear relation for  $v$-adic MZV's.

\end{abstract}


\maketitle
\section{Introduction}
\subsection{A conjecture of Furusho}
Let $\NN$ be the set of positive integers. Recall that the classical multiple zeta values are defined for $\fs=(s_{1},\ldots,s_{r})\in \NN^{r}$ with $s_{1} \geq 2$,
\[\zeta(\fs):=\sum_{n_{1}>n_{2}>\cdots>n_{r}\geq 1}\frac{ 1  }{n_{1}^{s_{1}}\cdots n_{r}^{s_{r}}  } \in \RR^{\times}.\] The weight and depth of the presentation $\zeta(\fs)$ are defined by $\wt(\fs):=\sum_{i=1}^{r}s_{i}$ and $\dep(\fs):=r$ respectively. In the past two decades, there have been a considerable amount of interest and vast developments in the topic of MZV's, which have
arisen in various contexts in number theory and arithmetical algebraic
geometry, etc. It is known that there are many $\QQ$-linear relations among the MZV's produced by the regularized double shuffle relations in~\cite{IKZ06}, but their exact structure remains mysterious. We refer the reader to the books~\cite{An04, Zhao16, BGF19}.

Let $p$ be a prime number. We first briefly review Furusho's $p$-adic MZV's in~\cite{F04}. Fix an $r$-tuple $\fs=(s_{1},\ldots,s_{r})\in \NN^{r}$ with $s_{1}\geq 2$ and note that the MZV $\zeta(\fs)$ is a specialization at $1$ of the one-variable  multiple polylogarithm
\[\Li_{(s_{1},\ldots,s_{r})}(z) :=\sum_{n_{1}>n_{2}>\cdots>n_{r}\geq 1}\frac{ z^{n_{1}}  }{n_{1}^{s_{1}}\cdots n_{r}^{s_{r}}  },   \]
which are generalizations of the classical logarithms.  Furusho considered the one-variable $p$-adic multiple polylogarithm $\Li_\fs(z)_p$, which is the same power series as $\Li_{\fs}(z)$, but treated $p$-adically.  This function converges on the open unit disk centered at $0$ of $\CC_{p}$, where $\CC_{p}$ is the $p$-adic completion of a fixed algebraic closure of $\QQ_{p}$. We note that the open unit disk centered at $0$ of $\CC_{p}$ and the one centered at $1$ of $\CC_{p}$ are disjoint, and so it does not make sense when taking limit $z\rightarrow 1$ on $\CC_{p}$. However, Furusho applied Coleman's $p$-adic iterated integration theory~\cite{Co82}  to make an analytic continuation of $p$-adic multiple polylogarithms and then defined the $p$-adic MZV $\zeta_{p}(\fs)$ to be  a certain limit value at 1 of the analytically continued $p$-adic multiple polylogarithm. The weight and depth of the presentation of the $p$-adic MZV $\zeta_{p}(\fs)$ are defined by $\wt(\fs)$ and $\dep(\fs)$ respectively.  Note that in the case of depth one, Furusho's $p$-adic zeta value $\zeta_{p}(s)$ is identical to the Kubota-Leopoldt $p$-adic zeta value at $s$ up to a scalar multiplication by $(1-p^{-s})^{-1}$.

In \cite{FJ07}, it was shown that the $p$-adic MZV's satisfy the regularized double shuffle relations~\cite{IKZ06}. One  precise connection between real-valued MZV's and $p$-adic MZV's is conjectured by Furusho: the $p$-adic MZV's satisfy the same $\QQ$-linear relations that their corresponding real-valued MZV's satisfy.
\begin{conjecture}[\textnormal{Furusho}\label{Conj: Furusho}]
Let $p$ be a prime number. Let $n\geq 2$ be an integer and  let $\mathfrak{Z}_{n}$ be the $\QQ$-vector space spanned by all real-valued MZV's of weight $n$, and $\mathfrak{Z}_{n,p}$ be the $\QQ$-vector space spanned by all $p$-adic MZV's  of weight $n$. Then we have a well-defined surjective $\QQ$-linear map
\[\mathfrak{Z}_{n} \twoheadrightarrow \mathfrak{Z}_{n,p} \]
given by 
\[\zeta(\fs)\mapsto \zeta_{p}(\fs) .\]
\end{conjecture}

The conjecture above is implicit in the two papers \cite{F06, F07}. Considering the graded algebra $\mathfrak{Z}:=\mathbb{Q} \oplus \bigoplus_{n\geq 2}\mathfrak{Z}_{n}$ (resp. $\mathfrak{Z}_p:=\mathbb{Q} \oplus \bigoplus_{n\geq 2}\mathfrak{Z_{n,p}}$), Furusho \cite[Conj.~A]{F06} conjectured that $\mathcal{O}({\underline{GRT}}_1)$ is isomorphic to $\mathfrak{Z}/(\pi^2)$ and in \cite[Sec.~3.1]{F07}  he explained that
there is a surjection from $\mathcal{O}({\underline{GRT}}_1)$ to $\mathfrak{Z}_p$. Here ${\underline{GRT}}_1$ is the unipotent part of the Grothendieck-Teichm\"uller group ${\underline{GRT}}$, which is a pro-algebraic group over $\mathbb{Q}$. For more details, see~\cite{F06, F07}.

Note that the Ihara-Kaneko-Zagier conjecture~\cite{IKZ06} asserts that the regularized double shuffle relations generate all the $\QQ$-linear relations among real-valued MZV's. However, the IKZ conjecture is a very difficult problem in classical transcendence theory and  it would imply Conjecture~\ref{Conj: Furusho} because of the result in~\cite{FJ07}. To date,  Conjecture~\ref{Conj: Furusho} is still open. In this paper, we come up with new ideas through logarithmic points of view to verify a stronger version of Furusho's conjecture in the setting of function fields in positive characteristic.

\subsection{The main result}
  Let $A:=\FF_{q}[\theta]$ be the polynomial ring in the variable $\theta$ over the finite field $\FF_{q}$ of $q$ elements with characteristic $p$.  Let $k$ be the fraction field of $A$ equipped with the normalized absolute value $|\cdot|_{\infty}$ associated with the infinite place $\infty$ for which $|\theta|_{\infty}=q$. Let $k_{\infty}$ be the completion of $k$ with respect to $|\cdot|_{\infty}$, and let $\CC_{\infty}$ be the completion of a fixed algebraic closure of $k_{\infty}$ with respect to the canonical absolute value extending  $|\cdot|_{\infty}$ on $k_{\infty}$. Let $\ok$ be the algebraic closure of $k$ in $\CC_{\infty}$.

In what follows, we will  review the $\infty$-adic multiple zeta values initiated by Thakur~\cite{T04}.  For a finite place $v$ of $k$, we will define one kind of $v$-adic multiple zeta values and abbreviate them as \lq\lq $v$-adic MZV's \rq\rq. To distinguish between the $\infty$-adic and $v$-adic settings, throughout this paper  we will  use \lq\lq MZV's\rq\rq~for Thakur's $\infty$-adic multiple zeta values unless we state the contrary. We further  mention that Thakur~\cite[p.~196]{T04} defined $v$-adic MZV's  using  Goss-Kummer congruences but his definition of $v$-adic MZV's is different from ours.

Fixing any $r$-tuple $\fs=(s_{1},\ldots, s_{r})\in \NN^{r}$, Thakur~\cite{T04} defined the following positive characteristic MZV's:
\begin{equation}\label{E:Def of MZV's}
\zeta_{A}(\fs):=\sum \frac{1}{a_{1}^{s_1}\cdots a_{r}^{s_r}} \in k_{\infty}.
\end{equation}
Here $a_{1},\ldots,a_{r}$ run over all monic polynomials in $A$ satisfying the strict inequalities: $|a_1|_{\infty}> |a_{2}|_{\infty}>\cdots >|a_r|_{\infty}$. Note that since our absolute value $|\cdot|_{\infty}$ is non-archimedean, the series $\zeta_{A}(\fs)$ converges $\infty$-adically in $k_{\infty}$ for all $\fs\in \NN^{r}$. Furthermore, it is shown by Thakur~\cite{T09} that every $\zeta_{A}(\fs)$ is non-vanishing. We call $\tiny{\rm{wt}}(\fs):=\sum_{i=1}^{r}s_{i}$ the weight and ${\tiny{\rm{dep}} (\fs)}:=r$ the depth of the presentation of $\zeta_{A}(\fs)$. Depth one MZV's were introduced by Carlitz~\cite{Ca35} and so are called Carlitz zeta values.  

Unlike the simple identity between real-valued MZV's and the specialization at $1$ of multiple polylogarithms in the classical theory, in the function field setting for any tuple $\fs\in \NN^{r}$ there are some explicit constants $b_{\ell}\in A$,  tuples of indices $\fs_{\ell}$ with $\wt(\fs_{\ell})=\wt(\fs), \dep(\fs_{\ell})\leq \dep(\fs)$ and integral points $\bu_{\ell}\in A^{\dep(\fs_{\ell})}$ so that  $\zeta_{A}(\fs)$ can be expressed as  the following formula (see Theorem~\ref{T:MZV as Lis}):
\[  \zeta_{A}(\fs)=\frac{1}{ \Gamma_{s_{1}} \cdots \Gamma_{s_{r}} } \sum_{\ell}b_{\ell} \cdot (-1)^{\tiny{\rm{dep}(\fs_{\ell})}-1} \Lis_{\fs_{\ell}}(\bu_{\ell}).\]
Here $\Gamma_{s_{i}}$ stands for the Carlitz factorials in $A$ given in (\ref{E:Def of Gamma}) and $\Lis_{\fs_{\ell}}$ is the $\fs_{\ell}$th Carlitz multiple star polylogarithm (in several variables), abbreviated as CMSPL and defined in~(\ref{E:Def of CMSPL}). Note that this identity was due to Anderson-Thakur~\cite{AT90} in the depth one case, where CMSPL's are reduced to Carlitz polylogarithms. For more details, see Sec.~\ref{Sec: Formulae for MZV's via CMPL's}.

Fix a finite place $v$ of $k$ and let $k_{v}$ be the completion of $k$ at $v$. Let $\CC_{v}$ be the $v$-adic completion of a fixed algebraic closure of $k_{v}$. In analogy with Kubota-Leopoldt's $p$-adic zeta function, Goss~\cite{Go79}  defined a $v$-adic zeta function that interpolates Carlitz zeta values at non-positive integers and obtained $v$-adic zeta values at positive integers, which are simply called {\it{ Goss' $v$-adic zeta values}}.

Our definition of $v$-adic MZV's is inspired by Furusho's definition of $p$-adic MZV's. We first note that the CMSPL $\Lis_{\fs_{\ell}}$ converges $v$-adically on the open unit ball centered at the zero of $\CC_{v}^{\dep(\fs_{\ell})}$. In \cite[Sec.~4]{CM17},  by using twists of certain $t$-module actions we showed that  $\Li_{\fs_{\ell}}^{\star}$ can be analytically continued to the closed unit ball centered at the zero of  $\CC_{v}^{\dep(\fs_{\ell})}$, and so it can be evaluated $v$-adically at $\bu_{\ell}$ and we denote this value by $\Li_{\fs_{\ell}}^{\star}(\bu_{\ell})_{v}$.  This leads naturally to the definition of $v$-adic MZV $\zeta_{A}(\fs)_{v}$ parallel to Furusho's approach in Definition~\ref{Def: v-adic MZV}:
\[  \zeta_{A}(\fs)_{v}:=\frac{1}{  \Gamma_{s_{1}} \cdots \Gamma_{s_{r}}  } \sum_{\ell}b_{\ell}\cdot (-1)^{\tiny{\rm{dep}(\fs_{\ell})}-1} \Lis_{\fs_{\ell}}(\bu_{\ell})_{v} \in k_{v}.\] The weight and depth of this presentation $\zeta_{A}(\fs)_{v}$ are defined to be $\wt(\fs)$ and $\dep(\fs)$ respectively.  

Note that here we do not exclude the $v$-part and so for each $s \in \NN$, our $v$-adic zeta value $\zeta_{A}(s)_{v}$ is identical to Goss' $v$-adic zeta value~\cite{Go79} at $s$ multiplied by $(1-v^{-s})^{-1}$ (see \cite[Theorem 3.8.3. (II)]{AT90}). This phenomenon is surprisingly parallel to the $p$-adic case mentioned above.  The main theorem of this paper stated as Theorem~\ref{Thm: Well-defined map} is to establish a stronger version of a function field analogue of Conjecture~\ref{Conj: Furusho} in the following result as it holds over algebraic coefficients.

\begin{theorem}[Theorem~\ref{Thm: Well-defined map}]\label{Thm: Well-defined map in Introduction} Let $v$ be a finite place of $k$ and fix an embedding $\ok \hookrightarrow \CC_{v}$. Let $n$ be a positive integer and  let $\overline{ \cZ}_{n}$ be the $\ok$-vector space spanned by all $\infty$-adic MZV's of weight $n$, and $\overline{\cZ}_{n,v}$ be the $\ok$-vector space spanned by all $v$-adic MZV's  of weight $n$. Then we have a well-defined surjective $\ok$-linear map
\[\overline{\cZ}_{n} \twoheadrightarrow \overline{\cZ}_{n,v} \]
given by 
\[\zeta_{A}(\fs)\mapsto \zeta_{A}(\fs)_{v} \]
and its kernel contains the one-dimensional vector space $\ok\cdot \zeta_{A}(n)$ when $n$ is divisible by $q-1$.
\end{theorem}

Since by \cite{Go79} we have $\zeta_{A}(n)_{v}=0$ when $n$ is divisible by $q-1$, an interesting consequence of the theorem above is that if $\zeta_{A}(\fs)$ is \lq\lq Eulerian\rq\rq, ie., $\zeta_{A}(\fs)$ is a $k$-multiple of the $\wt(\fs)$th power of the Carltiz period, then $\zeta_{A}(\fs)_{v}=0$. We mention that for fixed weight, the spirit of the result above is that $v$-adic MZV's  satisfy the same $\ok$-linear relations as their corresponding $\infty$-adic MZV's satisfy. For example, Thakur~\cite[Thm.~5]{T09b}  showed that $\zeta_{A}(m,m(q-1))=\zeta_{A}(mq)/ (\theta-\theta^{q})^{m}$ for $m\leq q$, and so by Theorem~\ref{Thm: Well-defined map in Introduction} we have $\zeta_{A}(m,m(q-1))_{v}=\zeta_{A}(mq)_{v}/(\theta-\theta^{q})^{m}$. Note that in~\cite{C14}, we have that  all $\ok$-linear relations among MZV's are generated by $k$-linear relations among MZV's of the same weight. This is Baker's phenomenon for MZV's over function fields and it  is out of reach for the classical MZV's as it is still very difficult to prove transcendence results in the classical case.

Finally, we mention that Thakur~\cite{T10} showed that  $\infty$-adic MZV's form a $k$-algebra. From numerical evidence using Sage, it seems that our $v$-adic MZV's form a $k$-algebra and the map from the space of MZV's to the space of $v$-adic MZV's is a $k$-algebra homomorphism. These questions are listed at the end of this paper and  we plan to tackle them in a future project.

\begin{remark}
Let $\fs=(s_{1},\ldots,s_{r})\in \NN^{r}$ and suppose that $\zeta_{A}(\fs)$ is  \lq\lq Eulerian\rq\rq. By~\cite[Cor.~4.2.3]{CPY14}, we have that each $s_{i}$ is divisible by $q-1$ for $i=1,\ldots,r$, and that each of the MZV's
\[ \zeta_{A}(s_{2},\ldots,s_{r}),\zeta_{A}(s_{3},\ldots,s_{r})\ldots,\zeta_{A}(s_{r})   \]
is \lq\lq Eulerian\rq\rq. Therefore, under the assumption  that $\zeta_{A}(\fs)$ is  \lq\lq Eulerian\rq\rq \ we have the simultaneous vanishing of the $v$-adic MZV's
\[ \zeta_{A}(s_{1},\ldots,s_{r})_{v}=\zeta_{A}(s_{2},\ldots,s_{r})_{v}=\cdots=\zeta_{A}(s_{r})_{v}=0  \]
by Theorem~\ref{Thm: Well-defined map in Introduction}, but we do not know whether the reverse direction is valid.  This phenomenon matches with one direction of  the criterion~\cite[Cor.~5.1.3]{CM17} for when Carlitz multiple polylogarithms at algebraic points are \lq\rq Eulerian\rq\rq.
\end{remark}

\subsection{Strategy of proof} In this section, we describe our strategy in proving Theorem~\ref{Thm: Well-defined map}. We first recall W\"ustholz's analytic subgroup theorem in  classical transcendence theory. 

\begin{theorem}[W\"ustholz~\cite{W89}]
Let $G$ be a connected commutative algebraic group defined over $\overline{\QQ}$. Let $\exp_{G}$ be the exponential map of $G$ when regarding $G(\CC)$ as a Lie group. Let ${\bu}\in \Lie G(\CC)$ satisfy $\exp_{G}(\bu)\in G(\overline{\QQ})$, and put $T_{\bu}$ to be the smallest linear subspace of $\Lie G (\CC)$ that is defined over $\overline{\QQ}$ and that contains the vector $\bu$. Then $T_{\bu}=\Lie H$ for some algebraic subgroup $H$ of $G$ that is defined over $\overline{\QQ}$. 
\end{theorem}

Note that W\"ustholz's theorem is one of the most powerful tools in classical transcendence theory when tackling transcendence question about generalized logarithms. The spirit of the theorem above is to assert that the $\overline{\QQ}$-linear relations among the coordinates of the generalized logarithm $\bu$ of an algebraic point arise from the defining equations of $\Lie H$ over $\overline{\QQ}$. One can derive Baker's celebrated theorem on linear forms in logarithms of algebraic numbers as well as its elliptic analogue using this  analytic subgroup theorem of W\"ustholz. See~\cite{BW07}.

In the function field setting, we have an analogue of W\"ustholz's theorem, called Yu's sub-$t$-module theorem  (see Sec.~\ref{Subsec: Yu's sub-t-module thm} for related definitions).

\begin{theorem}[{\cite[Thm.~0.1]{Yu97}}]\label{Thm: Yu Introduction}
Let $G$ be a regular t-module defined over $\ok$. Let $Z$ be a vector in $\Lie G(\CC_{\infty})$ such that $\exp_{G}(Z)\in G(\ok)$. Then the smallest linear subspace in $\Lie G(\CC_{\infty})$ defined over $\ok$, which is invariant under $\partial[t]$ and contains $Z$, is the tangent space at the origin of a sub-$t$-module $H$ of $G$ over $\ok$. 
\end{theorem}

The key ideas in proving Theorem~\ref{Thm: Well-defined map in Introduction} arise from the following observation. Suppose that we have $G, Z, H$ given in Theorem~\ref{Thm: Yu Introduction}. We put $\bv:=\exp_{G}(Z)\in H(\bar{k})$ and assume that $\log_{G}$ converges $v$-adically at $\bv$ denoted by $\log_{G}(\bv)_v$ and suppose further that  $\log_{G}(\bv)_v$ lies in $\Lie H(\CC_{v})$.  Regarding $\Lie H$ as a linear subvariety  of $\mathbb{A}^{\dim G}$ over $\ok$, we see that the coordinates of $Z$ and $\log_{G}(\bv)_{v}$ satisfy the defining equations of $\Lie H$ over $\ok$ since $Z\in \Lie H(\CC_{\infty})$ and $\log_{G}(\bv)_{v}\in \Lie H(\CC_{v})$. Therefore, to prove Theorem~\ref{Thm: Well-defined map in Introduction} our aim is to create suitable $G, Z, \bv, H$ as above and then relate the MZV's (resp. $v$-adic MZV's) in question to \lq suitable\rq~ coordinates of $Z$ (resp. corresponding coordinates of $\log_{G}(\bv)_v$). This novel method fully illustrates the spirit of the result in Theorem~\ref{Thm: Well-defined map in Introduction} through logarithms in a robust way, which simultaneously deals with both $\infty$-adic and $v$-adic special values and which is completely different from the point of view in the theory of classical MZV's. We believe that this kind of strategy via $\Lie H$ will explain more fundamental phenomena in a broader context.

\subsection{Logarithmic interpretations for MZV's}  In the seminal paper~\cite{AT90}, Anderson and Thakur gave  logarithmic interpretations for Carlitz zeta values and $v$-adic Goss zeta values, where $v$ is a finite place of $k$. For each Carlitz zeta value $\zeta_{A}(s)$, we consider the $s$th tensor power of the Carlitz module denoted by $(\bC^{\otimes {s}}, [\cdot]_s)$ in (\ref{E: Def of C otimes s}). Anderson and Thakur explicitly constructed a special point $\bv_{s}\in \bC^{\otimes {s}}(k)$ and a vector $Z_{s}\in \Lie \bC^{\otimes s}(\CC_{\infty})$ so that up to an explicit multiple $\Gamma_{s}$ in $A$, $\zeta_{A}(s)$ occurs as the $s$th coordinate of $Z_{s}$ and $\exp_{\bC^{\otimes s}}(Z_{s})=\bv_{s}$, where $\exp_{\bC^{\otimes s}}$ is the exponential map of $\bC^{\otimes s}$ (see Sec.~\ref{sec:Anderson t-modules}). It is further shown that the logarithm of $\bC^{\otimes s}$ converges $v$-adically at $[a]_{s}\bv_s$ for some nonzero $a\in \FF_{q}[t]$ and the last coordinate of this $v$-adic logarithmic vector gives $a(\theta)\Gamma_{s}\zeta_{A}(s)_{v}$.

Around that time, Yu~\cite{Yu91} developed a transcendence theory for the last coordinate of the logarithm of $\bC^{\otimes s}$ at algebraic points.  As an important consequence, Yu combined his work with that of
Anderon-Thakur to derive  the transcendence of all Carlitz zeta values as well as transcendence of $\zeta_{A}(s)_{v}$ for positive integers $s$ not divisible by $q-1$. Yu's transcendence results surpass the classical situation. Later on, Yu extensively generalized the $\infty$-adic transcendence theory in~\cite{Yu91} to the most general setting stated as Theorem~\ref{Thm: Yu Introduction}.  The transcendence of arbitrary MZV was obtained by the first author of the present paper~\cite{C14} using the current $t$-motivic transcendence theory, in particular the so-called ABP criterion, developed by Anderson, Brownawell and Papanikolas~\cite{ABP04}.

In \cite[Thm.~4.1.1]{C16}, the first author of the present paper gave a logarithmic interpretation for those MZV's $\zeta_{A}(s_{1},\ldots,s_{r})$ which have the property that $\zeta_{A}(s_{2},\ldots,s_{r})$ is {\it{Eulerian}}. The following two theorems are complete generalizations of Anderson-Thakur's work described above to arbitrary $\infty$-adic MZV's and $v$-adic MZV's (see Sec.~\ref{sec:Anderson t-modules} for the related definitions).

\begin{theorem}[Theorem~\ref{Thm: Infty-adic LogMZV}]\label{Thm: Introduction}
Given any $r$-tuple $\fs=(s_{1},\ldots,s_{r})\in \NN^{r}$, we put $n:=\wt(\fs)$. We explicitly construct a uniformizable $t$-module $G_{\fs}$ that is defined over $k$, a special point $\bv_{\fs} \in G_{\fs}(k)$ and a vector $Z_{\fs}\in \Lie G_{\fs}(\CC_{\infty})$ so that 
\begin{itemize}
\item[(a)] $\Gamma_{s_{1}}\cdots \Gamma_{s_{r}}\zeta_{A}(\fs)$ occurs as the $n$th coordinate of $Z_{\fs}$. 
\item[(b)] $\exp_{G_{\fs}}(Z_{\fs})=\bv_{\fs}$.
\end{itemize}
\end{theorem}
\begin{theorem}[Theorem~\ref{T: LogInt for v-adic MZV}]\label{T: v-adic MZV Intro} Fix a finite place $v$ of $k$ and let notation be given in Theorem~\ref{Thm: Introduction}. We  take a nonzero $a \in \FF_{q}[t]$ for which $|[a] \bv_{\fs}|_{v} < 1$.  Then the $n$th coordinate of $\log_{G_{\fs}}([a] \bv_{\fs})_{v}$ is given  by $a(\theta)\Gamma_{\fs}\zeta_{A}(\fs)_{v}$.
\end{theorem}

With the two theorems above, one is naturally lead to the naive question whether classical  MZV's can be interpreted via generalized logarithms at algebraic points fitting into Wuestholz's analytic subgroup theorem. However, the category of mixed Tate motives in characteristic zero  is certainly much more subtle than our category of t-motives in positive characteristic. Therefore the dream of connecting classical MZV's with $p$-adic MZV's is currently far beyond reach.

As mentioned above, the result of Theorem~\ref{Thm: Introduction} for the depth one case was established by Anderson-Thakur~\cite{AT90}. However, we have not been able to find an easier way to generalize their methods to the higher depth case. In \cite{AT90}, there are two crucial points in the scheme of their proof:

\begin{enumerate}
\item[(1)] Interpolation of power sums (see~\cite[(3.7.4)]{AT90}).
\item[(2)] Formulas for the right lower corner of coefficient matrices of the logarithm of $\bC^{\otimes s}$ (see~\cite[Prop.~2.1.5]{AT90}).
\end{enumerate}
Property $(1)$ enables one to connect $\zeta_{A}(s)$ with a $k$-linear combination of the $s$th Carlitz polylogarithm at certain integral points, and $(2)$ allows one to express the last coordinate of  the logarithm of $\bC^{\otimes s}$ at a specific special point as an evaluation of the $s$th Carlitz polylogarithm.  Note that the interpolation property $(1)$ was used by the first author of the present paper to express each MZV $\zeta_{A}(\fs)$ as a $k$-linear combination of the $\fs$th Carlitz multiple polylogartihm at integral points~\cite{C14}. 

Inspired by the period interpretation of MZV's in \cite{AT09}, for each MZV $\zeta_{A}(\fs)$ the authors of \cite{CPY14} constructed a $t$-module $E_{\fs}'$ defined over $A$ and a special point $\bv_{\fs}\in E_{\fs}'(A)$ to establish a criterion in terms of $(E_{\fs}',\bv_{\fs})$ for determining when $\zeta_{A}(\fs)$ is a $k$-rational multiple of $\zeta_{A}(\wt(\fs))$. It is natural to ask or predict whether $\zeta_{A}(\fs)$ can be connected to the logarithm of $E_{\fs}'$.  The difficulty along this direction is that in general the $t$-module $E_{\fs}'$ is complicated, and so far we do not know how to spell out a rule of writing it down explicitly except case by case (see~\cite[Sec.~6.1.1]{CPY14}). Therefore, it is difficult to compute the coefficient matrices of the logarithm of $E_{\fs}'$ following Anderson-Thakur's methods, which involve recursive matrix calculations.  For other instances involving calculations of the logarithm of a higher dimensional $t$-module, see~\cite{CM17, G17}.

To circumvent the difficulty mentioned above, we introduce new techniques based on fiber coproducts of Anderson dual $t$-motives  and we sketch the ideas to prove Theorem~\ref{Thm: Introduction} below. Fix an $r$-tuple $\fs=(s_{1},\ldots,s_{r})\in \NN^{r}$ with $n:={\tiny{\rm{wt}}}(\fs)$.

\begin{itemize}
\item[(I).] Based on the formula~\cite[Thm.~5.5.2]{C14} we further express $\Gamma_{s_{1}}\cdots \Gamma_{s_{r}}\zeta_{A}(\fs)$ as an explicit $A$-linear combination of  CMSPL's at some integral points. See Theorem~\ref{T:MZV as Lis}.
\item[(II).] For each triple $(b_{\ell},\fs_{\ell},\bu_{\ell})$ occurring in the right hand side of the identity in Theorem~\ref{T:MZV as Lis}, following~\cite{CM17} we explicitly construct a {\it{uniformizable}} $t$-module $G_{\ell}$ defined over $k$ and a special point $\bv_{\ell}\in G_{\ell}(k)$ and show that the logarithm $\log_{G_{\ell}}$ of $G_{\ell}$ converges at the special point $\bv_{\ell}$, and that the $n$th coordinate of $\log_{G_{\ell}}(\bv_{\ell})$ gives $(-1)^{\tiny{\rm{dep}}(\fs_{\ell})-1}\Lis_{\fs_{\ell}} (\bu_{\ell})$. See Theorem~\ref{theorem_log_li}.
 \item[(III).] We mention that $G_{\ell}$ comes from a rigid analytically trivial Anderson dual $t$-motive $\cM_{\ell}'$ with $C^{\otimes n}$ as a sub-$t$-motive, where $C^{\otimes n}$ is the $n$th tensor power of the Carlitz $t$-motive (see Remark~\ref{Rem: R.A.T of M'}). We then define $\cM$ to be the fiber coproduct of those $\cM_{\ell}'$ over  $C^{\otimes n}$ and show that it is rigid analytically trivial in Proposition~\ref{Prop: R.A.T}. Such $\cM$ corresponds to a uniformizable $t$-module $G_{\fs}$ and one has a natural morphism $\pi: \oplus_{\ell} G_{\ell}\rightarrow G_{\fs}$ defined over $k$ (see~Lemma~\ref{L:Key Lemma}). We then define $Z_{\ell}:=\log_{G_{\ell}}(\bv_{\ell})$, $Z_{\fs}:=\partial \pi \left( (\partial[b_{\ell}(t)] Z_{\ell})_{\ell}\right)\in \Lie G_{\fs} (\CC_{\infty})$, and $\bv_{\fs}:=\pi \left( ([b_{\ell}(t)] \bv_{\ell})_{\ell}\right)\in G_{\fs}(k)$, where $b_{\ell}\in A$ are given in Theorem~\ref{T:MZV as Lis} and $\partial[\cdot]$ is given in (\ref{E:partial rho a}).
\item[(IV).] In Lemma~\ref{Lemma: Key Lemma}, we  show that the $n$th coordinate of $Z_{\fs}$ is exactly the summation of the $n$th coordinate of $\partial[b_{\ell}(t)] Z_{\ell}$. Then by the formula in Theorem~\ref{T:MZV as Lis}  the desired result follows.
\end{itemize}
\subsection{Organization of this paper} We mention that one of our goals in writing this paper has been  to introduce our techniques in as general and robust a form as possible. Therefore we do not organize this paper in order matching the steps from (I) to (IV) above.   We first review the related theory of Anderson $t$-modules and Anderson dual $t$-motives in Section~\ref{Sec: Fiber coproducts}, and then consider the fiber coproducts of  Anderson dual $t$-motives in a setting as general as possible. The purpose of Section~\ref{Sec: Key Lemma} is to establish the key result in Lemma~\ref{Lemma: Key Lemma} for handling tractable coordinates of logarithmic vectors with respect to the fiber coproduct in question. Then step~(IV) above becomes a consequence of Lemma~\ref{Lemma: Key Lemma}. Section~\ref{Sec: log G ell} is devoted to verify Step~(II) above. In  Section~\ref{Sec: Proof of Main Thm} we set the stage for our MZV's: to any given MZV we associate a fiber coproduct family of Anderson dual $t$-motives satisfying the hypothesis of Lemma~\ref{Lemma: Key Lemma}; furthermore an explicit integral point is picked up on the $t$-module associated to each of the $t$-motives in this coproduct family. This set up then enables us to prove Theorem~\ref{Thm: Introduction}.

Finally, we define $v$-adic MZV's in Section~\ref{Sec: v-adic MZV} and prove~Theorem~\ref{T: v-adic MZV Intro}. We then use these logarithmic interpretations for $\infty$-adic and $v$-adic MZV's (Theorems~\ref{Thm: Introduction} and \ref{T: LogInt for v-adic MZV}) as well as Yu's sub-$t$-module theorem~\cite{Yu97} to prove Theorem~\ref{Thm: Well-defined map in Introduction}. At the end we list three natural and interesting questions in  Remark~\ref{Rem: future project}, which we will  investigate in a future project.

\subsection*{Acknowledgements}
The first author thanks D.~Zagier for helpful discussions that highly motived the formulation of Theorem~\ref{Thm: Well-defined map in Introduction}.  We are grateful to F.~Brown, D.~Brownawell, H.~Furusho, M.~Papanikolas, D.~Thakur, S.~Yasuda and  J.~Yu   for many helpful comments. We particularly thank the referee for careful reading and for many helpful suggestions and comments. The first author thanks Max Planck Institute for Mathematics for financial support and its hospitality. The project was initiated when the second author visited NCTS and we would like to thank NCTS for their kind support. This project was partially supported by JSPS Bilateral Open Partnership Joint Research Projects.

\section{Fiber coproduct of Anderson dual $t$-motives}\label{Sec: Fiber coproducts}
Throughout this paper, we will call  {\it{Anderson dual $t$-motives}} for those called {\it{dual $t$-motives}} in \cite{ABP04} and called {\it{Anderson $t$-motives}} in \cite{P08}.
\subsection{Anderson dual $t$-motives}  Let $\laurent{\CC_{\infty}}{t} $ be the field of Laurent series in the variable $t$ over $\CC_{\infty}$. For an integer $i$,  we define the $i$th fold twisting automorphism on $\laurent{\CC_{\infty}}{t}$ given by $f\mapsto f^{(i)}$, where 
$f^{(i)}:=\sum a_{j}^{q^{i}}t^{j} \hbox{ for }f=\sum a_{j} t^{j}\in \laurent{\CC_{\infty}}{t}.$ We extend the $i$th fold twisting to an operator on matrices with entries in $\laurent{\CC_{\infty}}{t}$ by entry-wise action.

We define the twisted polynomial ring $\ok[t,\sigma]$ generated by the two variables $t$ and $\sigma$ subject to the relations 
\[ \sigma f=f^{(-1)}\sigma \hbox{ for }f\in \ok[t]. \]

\begin{definition} An Anderson dual $t$-motive is a left $\ok[t,\sigma]$-module $\cM$ satisfying that
\begin{enumerate}
\item $\cM$ is a free left $\ok[t]$-module of finite rank.
\item $\cM$ is a free left $\ok[\sigma]$-module of finite rank.
\item $(t-\theta)^{s} \cM \subset \sigma \cM$ for all sufficiently large integers $s$.
\end{enumerate} 
\end{definition}
We note that the above notion of Anderson dual $t$-motives can be defined over any perfect field $L$ containing $k$, but for our purpose from the point of view of transcendence theory the field $\ok$ is the most suitable.
For an Anderson dual $t$-motive $\cM$ of rank $r$ over $\ok[t]$ and of rank $d$ over $\ok[\sigma]$, we call the vector $\bx=(x_{1},\ldots,x_{r})\in \Mat_{1\times r}(\cM)$ (resp. $\nu=(\nu_{1},\ldots,\nu_{d})\in \Mat_{1\times d}(\cM)$) a $\ok[t]$-basis (resp. a $\ok[\sigma]$-basis) for $\cM$ if $x_{1},\ldots,x_{r}$ (resp. $\nu_{1},\ldots,\nu_{d}$) form a $\ok[t]$-basis (resp. $\ok[\sigma]$-basis) of $\cM$. Fixing a $\ok[t]$-basis $\bx$ for $\cM$, then there exists a unique matrix $\Phi\in \Mat_{r}(\ok[t])\cap \GL_{r}(\ok(t))$ satisfying that
\[ \sigma \bx^{\tr}=\Phi \bx^{\tr} ,\]
where $\sigma \bx^{\tr}$ is defined via entry-wise action. We say that the matrix $\Phi$ represents multiplication by $\sigma$ on $\cM$ with respect to $\bx$ (cf.~\cite[Sec.~3.2.3]{P08}).

A typical example is the $n$th tensor power of the Carlitz $t$-motive denoted by $C^{\otimes n}$ for a positive integer $n$. The underlying module of $C^{\otimes n}$ is $\ok[t]$, on which $\sigma$ acts by
\[\sigma f:=(t-\theta)^{n} f^{(-1)}\hbox{ for }f\in \ok[t]. \]
It is not hard to check that $C^{\otimes n}$ is an Anderson dual $t$-motive with a $\ok[\sigma]$-basis given by
\[ \left( (t-\theta )^{n-1},\ldots,(t-\theta),1  \right). \]
As a left $\FF_{q}[t]$-module the quotient $C^{\otimes n}/(\sigma-1)C^{\otimes n}$ gives the $\ok$-valued points of the $n$th tensor power $\bC^{\otimes n}$ of the Carlitz module defined in (\ref{E: Def of C otimes s}). This fact was known by Anderson, and the reader can consult \cite{T04} and \cite[Sec.~5.2]{CPY14}.

\subsection{Anderson $t$-modules}\label{sec:Anderson t-modules}
We quickly review the theory of $t$-modules developed by Anderson in~\cite{A86}. For any field extension $L/k$, we let $\tau:L\rightarrow L$ be the Frobenius $q$th power operator, and one naturally extends it to an operator on $L^{s}$ by entry-wise action. Let $L[\tau]$ be the twisted polynomial ring generated by $\tau$ over $L$ subject to the relation:
\[ \tau \alpha=\alpha^{q} \tau \hbox{ for }\alpha\in L. \]
Given a $d$-dimensional additive algebraic group $ {\GG_{a}^{d}}_{/L}$ over $L$, we denote by  $\End_{\FF_{q}}\left( {\GG_{a}^{d}}_{/L} \right) $ the ring of endomorphisms of ${\GG_{a}^{d}}_{/L}$ that are $\FF_{q}$-linear and defined over $L$, and we naturally identify  $\End_{\FF_{q}}\left( {\GG_{a}^{d}}_{/L} \right) $ with the matrix ring $\Mat_{d}(L[\tau])$.

A $d$-dimensional $t$-module defined over $L$ is a pair $G=({\GG_{a}^{d}}_{/L}, \rho)$, where ${\GG_{a}^{d}}_{/L}$ is the $d$-dimensional additive group $\GG_{a}^{d}$ that is defined over $L$ and $\rho$ is an $\FF_{q}$-linear ring homomorphism 
\[ \rho:\FF_{q}[t] \rightarrow \End_{\FF_{q}}\left( {\GG_{a}^{d}}_{/L} \right)   \]
so that $\partial \rho_{t}-\theta I_{d}$ is a nilpotent matrix, where   $\partial\rho_{t}$  is defined to be the induced morphism of $\rho_{t}$ at the identity on the Lie algebra $\Lie{\GG_{a}^{d}}_{/L}$ of ${\GG_{a}^{d}}_{/L}$.  
 For a nonzero polynomial $a\in \FF_{q}[t]$,  we write 
$\rho_{a}=\sum_{i=0}^{m}a_{i}\tau^{i}$ with $a_{i}\in \Mat_{d}(L)$, where we understand that the symbols $a_i$ and $m$ depend on $a$. Then the differential of $\rho_{a}$ at the identity is explicitly expressed as
\begin{equation}\label{E:partial rho a}
\partial \rho_{a}=a_{0}.  
\end{equation}
Note that $G(F)=\GG_{a}^{d}(F)$ has a left $\FF_{q}[t]$-module structure via the map $\rho$ for any field extension $F/L$. 

Given such  a $d$-dimensional $t$-module $G$ over $L$, Anderson~\cite{A86} showed the existence of a $d$-variable power series $\exp_{G}$ with coefficients in $L$ for which
\begin{enumerate}
\item[(a)] $\exp_{G}(\bz)\equiv \bz \hbox{ } (\hbox{mod deg }q)$;
\item[(b)] for any $a\in \FF_{q}[t]$, the following identity holds:
 \begin{equation}\label{E:FEexp}
 \rho_{a}\circ \exp_{G}=\exp_{G}\circ \partial \rho_{a}.
 \end{equation}
\end{enumerate}
We mention that when we work over the field $\CC_{\infty}$,  $\exp_{G}:\Lie G(\CC_{\infty})\rightarrow G(\CC_{\infty})$ is entire. Such as the classical terminology for Lie groups, we call $\exp_{G}$ {\it{the exponential map}} of the $t$-module $G$.  The formal inverse of the power series $\exp_{G}$ is called the {\it{logarithm}} of $G$ denoted by $\log_{G}$ and it satisfies: 

\begin{equation}\label{E:IdentityExpLog}
\exp_G \circ \log_{G}(\bz)=\bz=\log_{G}\circ \exp_{G}(\bz) {\hbox{ (as power series identities)}}. 
\end{equation}
\begin{equation}\label{E:FuncE for log G}
\log_{G}\circ \rho_{a}=\partial \rho_{a}\circ \log_{G}  \hbox{ for every } a\in \FF_{q}[t].
\end{equation}
Note that $\log_{G}$ is the power series expansion around the origin of the multi-valued inverse map to $\exp_{G}$.

In fact, the exponential map $\exp_{G}$ is functorial in $G$ in the following sense. Let $G$ and $G'$ be two $t$-modules defined over $L$. By a morphism from $G$ to $G'$ over $L$, we mean a morphism as algebraic groups $\phi:G\rightarrow G'$ that is defined over $L$ and that commutes with $\FF_{q}[t]$-actions. The functoriality property~\cite[p.~473]{A86} means that we have the following functional equation:
\begin{equation}\label{E:ExpFunctorial}
\phi \circ \exp_{G}=\exp_{G'}\circ \partial \phi,
 \end{equation}
 where $\partial \phi$ is  the differential of the morphism $\phi$ at the identity. The functional equation for exponential maps and (\ref{E:IdentityExpLog}) imply the following functional equation for logarithms:
 \begin{equation}\label{E:logFunctorial}
 \log_{G'}\circ \phi=\partial \phi \circ \log_{G}. 
 \end{equation}

An example of a $t$-module is the $s$th tensor power of the Calitz module denoted by $\bC^{\otimes s}=({\GG_{a}^{s}}_{/k},[-]_{s})$ for any positive integer $s$. The underlying space of $\bC^{\otimes s}$ is ${\GG_{a}^{s}}_{/k}$ equipped with the $\FF_{q}[t]$-module structure given (and so uniquely determined) by 

\begin{equation}\label{E: Def of C otimes s}
 [t]_{s}=  \left(
\begin{array}{ccccc}
\theta& 1 & 0 & \cdots & 0 \\
& \theta & 1 & \ddots & \vdots \\
& & \ddots & \ddots & 0 \\
& & & \ddots & 1 \\
\tau& & & & \theta
\end{array}
\right)
\in \Mat_{s}( k[\tau] ).
\end{equation}

We call a $t$-module $G$ over $\ok$ {\it{uniformizable}} if its exponential map $\exp_{G}: \Lie G (\CC_{\infty})\rightarrow G(\CC_{\infty})$ is surjective. We mention that there are examples of $t$-modules which are not uniformizable, see \cite[Sec.~2.2]{A86}. Note that $\bC^{\otimes s}$ is uniformizable for each $s\in \NN$, see~\cite[Cor.~5.9.38]{Go96}.

\subsection{From Anderson dual $t$-motives to $t$-modules} Here we review how one constructs a $t$-module from an Anderson dual $t$-motive following Anderson's approach (see \cite[Sec.~5.2]{CPY14}, \cite[Sec.~4.4]{BP16} and \cite[Sec.~5.2]{HJ16}). Let $\cM$ be an Anderson dual $t$-motive with a $\ok[t]$-basis  $\left(x_{1},\ldots,x_{r}\right)$, and a $\ok[\sigma]$-basis $\left(\nu_{1},\ldots,\nu_{d} \right)$.  For any $y\in \cM$, we express $y=\sum_{i=1}^{d}g_{i}\nu_{i}$ with $g_{i}\in \ok[\sigma]$ and then define the map $\Delta:\cM\rightarrow \Mat_{d\times 1}(\ok)$ by
\begin{equation}\label{E:Def of Delta}
\Delta(y) :=(\delta(g_{1}),\ldots,\delta(g_{d}))^{\tr}\in \Mat_{d\times 1}(\ok), 
\end{equation}
where for $g=\sum_{j}a_{j}\sigma^{j} (=\sum_{j}\sigma^{j}a_{j}^{q^{j}} ) \in\ok[\sigma]$, $\delta:\ok[\sigma]\rightarrow \ok$ is defined by \[\delta(g):=\sum_{j} a_{j}^{q^{j}}.\] It is clear that $\Delta$ is $\FF_{q}$-linear and surjective. One further checks that $\Ker \Delta=(\sigma-1)\cM$, and therefore we have the induced isomorphism
\[\Delta: \cM/(\sigma-1)\cM \cong \Mat_{d\times 1}(\ok). \] As $\FF_{q}[t]$ is contained in the center of $\ok[t,\sigma]$,  $\cM/(\sigma-1)\cM$ has a left $\FF_{q}[t]$-module structure, which allows us to induce an $\FF_{q}[t]$-module structure on $\Mat_{d\times 1}(\ok)$  from the isomorphism above. One thereby has a unique $\FF_{q}$-linear ring homomorphism
\[\rho:\FF_{q}[t]\rightarrow \Mat_{d}(\ok[\tau]), \]
whence defining a $t$-module $G=({\GG_{a}^{d}}_{/ \ok},\rho)$ associated to the Anderson dual $t$-motive $\cM$ since the group of $\ok$-valued points is Zariski dense in ${\GG_{a}^{d}}_{/ \ok}$. 
\subsection{The fiber coproduct}\label{Subsec: Fiber Coproduct}

In this section, we will construct a fiber coproduct of certain Anderson dual $t$-motives, which will play the key role when proving Theorem~\ref{Thm: Introduction}. Here, we deal with the situation as general as possible, and expect it to have wide applications for the related issues.

\subsubsection{The set up}\label{sub: set up} Let $\cN$ be an Anderson dual $t$-motive of rank $r$ over $\ok[t]$, and we fix a $\ok[t]$-basis $\bx=(x_{1},\ldots,x_{r})\in \Mat_{1\times r}(\cN)$  as well as  a $\ok[\sigma]$-basis $\alpha=(\alpha_{1},\ldots,\alpha_{n})\in \Mat_{1\times n}(\cN)$ for $\cN$. Let $B:=B_{\cN}\in \Mat_{r}(\ok[t])\cap \GL_{r}(\ok(t))$ be the matrix presenting multiplication by $\sigma$ on $\cN$ with respect to $\bx$, ie.,
\[\sigma \bx^{\tr}=B \bx^{\tr}  .\]

Suppose that $\left\{ \cM_{\ell}' \right\}_{\ell=1}^{T}$ is a family of Anderson dual $t$-motives  equipped with the property that $\cM_{\ell}'$ contains $\cN$ as $\ok[t,\sigma]$-submodule for which either 

\begin{equation}\label{E: case1}
\cM_{\ell}' = \cN
\end{equation}
or 
\begin{equation}\label{E: case2}
\cM_{\ell}' \hbox{ fits into the short exact sequence of left }\ok[t,\sigma]\hbox{-modules}
\end{equation}

\[0 \rightarrow \cN \rightarrow \cM_{\ell}' \rightarrow \cM_{\ell}'' \rightarrow 0, \]
where $\cM_{\ell}''$ is an Anderson dual $t$-motive of rank $m_{\ell} \geq 1$ over $\ok[t]$.
We let $\cT=\left\{ 1,\ldots,T\right\}$ and decompose it as the disjoint union
\[\cT=\cT_{1} \cup \cT_{2},\]
where $\cT_{1}$ consists of those indexes $\ell$ for which $\cM_{\ell}'$ satisfies (\ref{E: case1}) and $\cT_{2}$ consists of those indexes $\ell$ for which $\cM_{\ell}'$ satisfies (\ref{E: case2}). We let $s:=|\cT_{1}|$, and for convenience we rearrange the indexes so that  
\[\cT_{1}=\left\{1,\ldots,s \right\}\hbox{ and }\cT_{2}=\left\{ s+1,\ldots,T\right\}. \]
It is allowed to be the case that $s=0$, ie., $\cT_{1}=\emptyset$ and $\cT_{2}=\cT$, or the case that $s=T$, ie., $\cT_{1}=\cT$ and $\cT_{2}=\emptyset$. In the latter case when $s=T$, it means that $\cM_{\ell}'$ is isomorphic to $\cN$ for all $\ell$. In the former case when $s=0$, every $\cM_{\ell}'$ is an extension of $\cM_{\ell}''$ by $\cN$ in (\ref{E: case2}).

For convenience  we put $m_{\ell}=0$ for $1\leq \ell \leq s$.
For each $1\leq \ell \leq T$, we denote by $\bx_{\ell}=(x_{\ell 1},\ldots, x_{\ell r})\in \Mat_{1\times r}(\cM_{\ell}')$ the image of the $\ok[t]$-basis $\bx = (x_{1}, \dots, x_{r})$ for $\cN$ under the map $\cN \hookrightarrow \cM_{\ell}'$.
Since $\cM_{\ell}''$ is free of rank $m_{\ell}$ over $\ok[t]$, there exist vectors $\by_{\ell}=(y_{\ell 1},\ldots, y_{\ell m_{\ell}})\in \Mat_{1\times m_{\ell}}(\cM_{\ell}')$, where $\by_{\ell}=\emptyset $ for $1\leq \ell \leq s$,
 so that $\left(\bx_{\ell},\by_{\ell} \right)$ is a $\ok[t]$-basis for $\cM_{\ell}'$.
For the $\ok[t]$-basis $(\bx_{\ell}, \by_{\ell})$, the action of $\sigma$ is given by the form
\[
\sigma \begin{pmatrix}
\bx_{\ell}^{\tr}\\
\by_{\ell}^{\tr}
\end{pmatrix} = 
\begin{pmatrix}
B & 0\\
D_{\ell} & \Phi_{\ell}''
\end{pmatrix}
\begin{pmatrix}
\bx_{\ell}^{\tr}\\
\by_{\ell}^{\tr}
\end{pmatrix}.
\] Here $\Phi_{\ell}''\in \Mat_{m_{\ell}}(\ok[t])\cap \GL_{m_{\ell}}(\ok(t))$ is the matrix representing multiplication by $\sigma$ on $\cM_{\ell}''$ with respect to the $\ok[t]$-basis as the image of $\by_{\ell}$ in $\cM_{\ell}''$. 

 For each $1\leq \ell \leq T$, we denote by $\widetilde{\alpha}_{\ell}:=(\alpha_{\ell 1},\ldots, \alpha_{\ell n})\in \Mat_{1\times n}(\cM_{\ell}')$ the image of the $\ok[\sigma]$-basis $\alpha=(\alpha_{1},\ldots,\alpha_{n})$ for $\cN$ under the map $\cN \hookrightarrow \cM_{\ell}'$. We understand that for $1\leq \ell \leq s$, $\widetilde{\alpha}_{\ell}$ is a $\ok[\sigma]$-basis for $\cM_{\ell}'$ since $\cN \cong \cM_{\ell}'$, and since $\cM_{\ell}''$ is free of finite rank over $\ok[\sigma]$ for $s+1\leq \ell \leq T$, $\widetilde{\alpha}_{\ell}$ can be extended to a $\ok[\sigma]$-basis $(\widetilde{\alpha}_{\ell},\beta_{\ell})$ for $\cM_{\ell}'$ for some $\beta_{\ell}\in \Mat_{1\times h_{\ell}}(\cM_{\ell}')$ with $h_{\ell} := \rank_{\ok[\sigma]} \cM_{\ell}''$. Note that the image of $\beta_{\ell}$ under the quotient map $\cM_{\ell}' \twoheadrightarrow \cM_{\ell}''$ forms a $\ok[\sigma]$-basis for $\cM_{\ell}''$. By convenience, for $1\leq \ell \leq s$ we put $h_{\ell}=0$ and $\beta_{\ell}=\emptyset$.

We note that to prove Theorem~\ref{Thm: Introduction}, we will take $\cN$ to be the $n$th tensor power of the Carlitz $t$-motive and take $\left\{ \cM_{\ell}' \right\}$ to be the Anderson dual $t$-motives constructed in~\cite{C14, CPY14}, whose periods involve Carlitz multiple polylogarithms at specific integral points. Related details are given in Sec.~\ref{Sec: proof of main thm}. 

\subsubsection{The definition}\label{sub:Def fiber coproduct} We continue with the notation and set up as above. We define $\cM$ to be the fiber coproduct of all $\cM_{\ell}'$ over $\cN$ denoted by \[\cM:=\cM_{1}'\sqcup_{\cN}\cM_{2}' \sqcup_{\cN} \cdots \sqcup_{\cN} \cM_{T}'.\] More precisely, as a left $\ok[t]$-module, $\cM$ is defined by the quotient:
\begin{equation}\label{E:QuotM}
 \cM:= \left( \oplus_{\ell=1}^{T}\cM_{\ell}' \right) \big/ \left( \Span_{\ok[t]} \left\{x_{\ell i}-x_{\ell' i}| 1\leq \ell , \ell' \leq T, \hbox{ }1\leq i \leq r \right\} \right) .
\end{equation}
Without confusion, we denote by $x_i$  the image of $x_{\ell i}$ in the quotient module $\cM$ for any $\ell$, and $1\leq i \leq r$. This is well defined from the description of $\cM$ above, and it makes sense to use the notation as one has the natural embedding $\cN\hookrightarrow \cM$. We still denote by $y_{\ell j}$ the image of $y_{\ell j}$ in the quotient module $\cM$ for $s+1\leq \ell \leq T$, and $1\leq j\leq m_{\ell}$, as it is well-defined due to (\ref{E:QuotM}).  Under such notation, it is clear to see that $\cM$ is a free $\ok[t]$-module and 
\begin{equation}\label{E:bm}
\bm:= \left( \bx, \by_{s+1},\ldots,\by_{T} \right) 
\end{equation}
is a $\ok[t]$-basis for $\cM$.

\begin{proposition}\label{Prop: t-motive M} The left $\ok[t]$-module $\cM$ defined above is an Anderson dual $t$-motive.
\end{proposition}
\begin{proof}
We first claim that the $\ok[t]$-submodule
$\Span_{\ok[t]} \left\{x_{\ell i}-x_{\ell' i}| 1\leq \ell , \ell' \leq T, \hbox{ }1\leq i \leq r \right\} $
is stable under the $\sigma$-action, whence a left $\ok[t,\sigma]$-submodule of $\oplus_{\ell =1}^{T}\cM_{\ell}'$. To show this, we note that $x_{\ell i}-x_{\ell' i}$ is the $i$th component of $(\bx_{\ell}-\bx_{\ell'})^{\tr}$. By  definition, $\sigma (\bx_{\ell}-\bx_{\ell'})^{\tr}$ is the vector
\[B\cdot (\bx_{\ell}-\bx_{\ell'})^{\tr} .\]
By expanding the above vector we see that 
\[ \sigma(x_{\ell i}-x_{\ell' i})\in\Span_{\ok[t]} \left\{x_{\ell i}-x_{\ell' i}| 1\leq \ell , \ell' \leq T, \hbox{ }1\leq i \leq r \right\} .\]

To show that $\cM$ is free of finite rank over $\ok[\sigma]$, we first note that the following matrix
\begin{equation}\label{E: Phi}
\Phi:= \begin{pmatrix}
B& & & \\
D_{s+1}& \Phi_{s+1}''& & \\
\vdots& & \ddots &\\
D_{T}&&& \Phi_{T}''
\end{pmatrix} 
\end{equation} is the matrix representing the action of $\sigma$ on $\cM$ with respect to the $\ok[t]$-basis $\bm$ given in (\ref{E:bm}). It follows that we have the following short exact sequence of Anderson dual $t$-motives:
\[0 \rightarrow \cN \rightarrow \cM \rightarrow \oplus_{\ell=s+1}^{T} \cM_{\ell}''\rightarrow 0 .\]
By hypothesis, each $\cM_{\ell}''$ is an Anderson  dual $t$-motive, so is $\oplus_{\ell=s+1}^{T} \cM_{\ell}''$. Since $\cN$ and $\oplus_{\ell=s+1}^{T} \cM_{\ell}''$ are Anderson dual $t$-motives, $\cM$ is a finitely generated $\ok[\sigma]$-module. By \cite[Prop.~4.3.4]{ABP04}, we know that $\ok[t]$-torsion submodule of $\cM$ is as same as the $\ok[\sigma]$-torsion submodule of $\cM$, and hence $\cM$ is free over $\ok[\sigma]$ since $\cM$ is a free left $\ok[t]$-module. 

Finally, one directly checks that $(t-\theta)^{i}\cM \subset \sigma \cM$ for sufficiently large integers $i$, whence $\cM$ is an Anderson dual $t$-motive. 
\end{proof}

\begin{remark}
Note that we can write down the ranks of $\cM$ over $\ok[t]$ and over $\ok[\sigma]$ respectively. Precisely, we have
\[\rank_{\ok[t]} \cM =\rank_{\ok[t]}\cN +\sum_{\ell=s+1}^{T}\left(  \rank_{\ok[t]}\cM_{\ell}' -\rank_{\ok[t]}\cN \right)  \]
and 
\[ \rank_{\ok[\sigma]} \cM =\rank_{\ok[\sigma]}\cN +\sum_{\ell=s+1}^{T}\left(  \rank_{\ok[\sigma]}\cM_{\ell}' -\rank_{\ok[\sigma]}\cN \right) .\] 
\end{remark}

\subsection{Rigid analytic trivialization} Let $\TT\subset \laurent{\CC_{\infty}}{t}$ be the subring consisting of power series that are convergent on the closed unit disk centered at the zero of $\CC_{\infty}$. More precisely, every element $f$ in $\TT$ is of the form $f=\sum_{i=0}^{\infty} a_{i}t^{i}$ with the property that $|a_{i}|_{\infty}\rightarrow 0$ as $i \rightarrow \infty$. We follow~\cite{ABP04, P08} to introduce the following terminology (cf.~\cite{A86}).

\begin{definition}
Let $M$ be an Anderson dual $t$-motive of rank $r$ over $\ok[t]$. Let $\Phi\in \Mat_{r}(\ok[t])\cap \GL_{r}(\ok(t))$ be the matrix representing multiplication by $\sigma$ on certain $\ok[t]$-basis for $M$.  We say that $M$ is rigid analytically trivial if there exists a matrix $\Psi\in \GL_{r}(\TT)$ so that 
\[\Psi^{(-1)}=\Phi \Psi. \] Such a $\Psi$ is called a rigid analytic trivialization of $\Phi$.
\end{definition}

\begin{remark}\label{Rem: R.A.T}
If an Anderson dual $t$-motive is rigid analytically trivial, then its associated $t$-module is uniformizable. See~\cite[Sec.~4.5]{BP16}  and \cite[Thm.~5.2.8]{HJ16}.
\end{remark}

\begin{proposition}\label{Prop: R.A.T}
Let $\cN$, $\left\{ \cM_{\ell}' \right\}_{\ell=1}^{T}$ be the Anderson dual $t$-motives given in Sec.~\ref{sub: set up} and suppose that all of them are rigid analytically trivial. Then so is the fiber coproduct $\cM$ of $\left\{ \cM_{\ell}'\right\}_{\ell=1}^{T}$ over $\cN$. 
\end{proposition}
\begin{proof}
We continue with the above notation that $B$ is the matrix representing multiplication by $\sigma$ on $\bx$ for $\cN$, and for each $s+1\leq \ell \leq T$,
 \[
 \begin{pmatrix}
B& \\
D_{\ell}& \Phi_{\ell}''\\
\end{pmatrix} 
\] is the matrix representing multiplication by $\sigma$ on $(\bx_{\ell},\by_{\ell})$ for $\cM_{\ell}'$. 

Since $\cN$ and $\cM_{\ell}'$ are rigid analytically trivial, there exist rigid analytic trivializations $Q$ and 
\[\begin{pmatrix}
Q&\\
R_{\ell}& \Psi_{\ell}''\\
\end{pmatrix} \]
for which $Q^{(-1)}=B Q$ and 
\[\begin{pmatrix}
Q& \\
R_{\ell}& \Psi_{\ell}''\\
\end{pmatrix} ^{(-1)}= \begin{pmatrix}
B& \\
D_{\ell}& \Phi_{\ell}''\\
\end{pmatrix}  \begin{pmatrix}
Q& \\
R_{\ell}& \Psi_{\ell}''\\
\end{pmatrix} .  \]
Since $\Phi$ given in (\ref{E: Phi}) is the matrix representing multiplication by $\sigma$ on $\bm$ for $\cM$, we put
\[ \Psi:= \begin{pmatrix}
Q& & & \\
R_{s+1}& \Psi_{s+1}''& & \\
\vdots& & \ddots &\\
R_{T}&&& \Psi_{T}''
\end{pmatrix}
 \] and find that $\Psi$ is a rigid analytic trivialization of $\Phi$. So the desired result follows.
\end{proof}

\section{The key lemma}\label{Sec: Key Lemma}  We continue with the setting and notation given in Sec.~\ref{sub: set up} and Sec.~\ref{sub:Def fiber coproduct}. As $\cM$ is a quotient of $\oplus_{\ell=1}^{T}\cM_{\ell}'$, we have the natural projection map $\mu:\oplus_{\ell =1}^{T}\cM_{\ell}' \twoheadrightarrow \cM$. In fact, according to the definition of $\cM$ we can write down the map $\mu$ explicitly as
\begin{equation}\label{E: mu}
\mu\left( ( \sum_{i=1}^{r}f_{\ell i}(t) x_{\ell i}+\sum_{j=1}^{m_{\ell}} f_{\ell j}(t) y_{\ell j}  )_{\ell}\right)= \sum_{i=1}^{r} \left(\sum_{\ell=1}^{T}  f_{\ell i}(t) \right)x_{i}+ \sum_{\ell =s+1}^{T} \sum_{j=1}^{m_{\ell}} f_{\ell j}(t) y_{\ell j}.
\end{equation}

According to the set up in Sec.~\ref{sub: set up} that $\alpha$ is identified with $\widetilde{\alpha}_{\ell}$ in $\cM_{\ell}'$, it follows that
\begin{equation}\label{E: Span Equality}
  \Span_{\ok[t]} \left\{x_{\ell i}-x_{\ell' i}| 1\leq \ell , \ell' \leq T, \hbox{ }1\leq i \leq r \right\} =  \Span_{\ok[\sigma]} \left\{\alpha_{\ell i}-\alpha_{\ell' i}| 1\leq \ell , \ell' \leq T, \hbox{ }1\leq i \leq n \right\} ,
 \end{equation}
 hence it is well-defined so that we can denote by $\alpha_{i}$ the image of $\alpha_{\ell i}$ for any $1\leq \ell\leq T$ and $1\leq i \leq n$. Note that such a fact can be also seen from the definition of fiber coproduct.

We denote by $\alpha_{i}$ the image of $\alpha_{\ell i}$ in $\cM$, by $\alpha$ the image of $\widetilde{\alpha}_{\ell}\in\Mat_{1\times n}(\cM_{\ell}')$ in $\Mat_{1\times n}(\cM)$,  and by $\beta_{\ell j}$ the image of $\beta_{\ell j}\in \cM_{\ell}'$ in $\cM$, which are well-defined by (\ref{E: Span Equality}) and the condition of $\ok[\sigma]$-basis  $(\widetilde{\alpha}_{\ell},\beta_{\ell})$ for $\cM_{\ell}'$. From the setting in Sec.~~\ref{sub: set up}, we see that $\left( \alpha, \beta_{s+1},\ldots,\beta_{T} \right)$ is a $\ok[\sigma]$-basis for $\cM$.

\subsection{The setting} For each $1\leq \ell \leq T$, we let $G_{\ell}$ be the $t$-module associated to the Anderson dual $t$-motive $\cM_{\ell}'$, i.e., we have the $\FF_{q}[t]$-module isomorphism
\[ G_{\ell}(\ok)\cong \cM_{\ell}'/(
\sigma-1) \cM_{\ell}'.\] To simplify the notation, we denote by $[-]$ the $\FF_{q}[t]$-action on any $t$-module without confusions. We denote by $H$ the $t$-module associated to the Anderson dual $t$-motive $\cN$. By our hypothesis that $\cN\cong \cM_{\ell}'$ for $1\leq \ell \leq s$ and the identification of $\ok[\sigma]$-bases $\alpha$ and $\widetilde{\alpha}_{\ell}$, $H$ is the $t$-module associated to $\cM_{\ell}'$ for $1\leq \ell \leq s$. 

By Proposition~\ref{Prop: t-motive M} we know that $\cM$ is an Anderson dual $t$-motive. We let $G$ be the $t$-module associated to $\cM$, ie., $G(\ok)\cong \cM/ (\sigma-1)\cM$ as $\FF_{q}[t]$-modules.  Recall that $(\alpha,\beta_{s+1},\ldots,\beta_{T})$ is a $\ok[\sigma]$-basis of $\cM$ and the rank of $\cM_{\ell}'$ over $\ok[\sigma]$ is $n+h_{\ell}$ for $s+1\leq \ell \leq T$. So the dimension of $G$ is
\begin{equation}\label{E:dimG}
 \dim G=n+h_{s+1}+\cdots+h_{T}  .
\end{equation}

\subsection{The main diagram}\label{subsec: main diagram}
\begin{definition}
Let $n$ be the rank of $\cN$ over $\ok[\sigma]$. For any integer $m\geq n$ and any vector $\bz=(z_{1},\ldots,z_{m})^{\tr}\in \CC_{\infty}^{m}$, we put 
\[\hat{\bz}:=\begin{pmatrix}
z_{1}\\
\vdots\\
z_{n}
\end{pmatrix}\hbox{ and }
\bz_{-}:=\begin{pmatrix}
z_{n+1}\\
\vdots\\
z_{m}
\end{pmatrix} \]
and  so $\bz$ is expressed as
\[\bz=\begin{pmatrix}
\hat{\bz}\\
\bz_{-}
\end{pmatrix} .\]
\end{definition}

\begin{definition}\label{Def:pi}
We define a morphism $\pi \colon \bigoplus_{\ell=1}^{T} G_{\ell} \to G$  of algebraic groups by
\[
\pi((\bz_{1}^{\tr}, \dots, \bz_{T}^{\tr})^{\tr}) = (\sum_{\ell=1}^{T} \hat{\bz}_{\ell }^{\tr}, \bz_{1-}^{\tr}, \dots, \bz_{T-}^{\tr})^{\tr}
.\]
\end{definition}

Recall that $\mu:\oplus_{\ell=1}^{T}\cM_{\ell}'\rightarrow \cM$ is the natural quotient map, which is a left $\ok[t,\sigma]$-module homorphism by (\ref{E: Span Equality}). Via $\mu$ we find from the following Lemma that $\pi$ is indeed a morphism of $t$-modules. 

\begin{lemma} \label{L:Key Lemma}Let notation be given as above. Then the following diagram
\[
\xymatrix{
\bigoplus_{\ell=1}^{T} \cM'_{\ell} \ar[r]^{\Delta} \ar[d]_{\mu} & \bigoplus_{\ell=1}^{T} G_{\ell}(\ok) \ar[d]^{\pi} \\
\cM \ar[r]^{\Delta} & G(\ok)
}
\]
 commutes. In particular, $\pi$ is a morphism of $t$-modules. 
\end{lemma}

\begin{proof}
Recall that for each $1\leq \ell \leq T$, $(\widetilde{\alpha}_{\ell},\beta_{\ell})$ is a $\ok[\sigma]$-basis for $\cM_{\ell}'$. Since the maps $\Delta, \pi$ and $\mu$ are additive, it suffices to show the commutativity of the diagram on elements of the form
\begin{equation}\label{E:case1}
\omega=f_{\ell i}(\sigma) \alpha_{\ell i}\in \cM_{\ell}' \hookrightarrow \oplus_{\ell=1}^{T}\cM_{\ell}' \hbox{ for }1\leq \ell \leq T, 1\leq i\leq n
\end{equation}
and
\begin{equation}\label{E:case2}
\omega=g_{\ell i}(\sigma)\beta_{\ell i}\in \cM_{\ell}'\hookrightarrow \oplus_{\ell=1}^{T}\cM_{\ell}' \hbox{ for }s+1\leq \ell \leq T, 1\leq i\leq h_{\ell} .
\end{equation}

Let $\omega=f_{\ell i}(\sigma) \alpha_{\ell i}$ be given in (\ref{E:case1}). We write
\[\Delta(\omega)=
\begin{array}{rclll}
\ldelim({8}{4pt}[] & 0 & \rdelim){8}{4pt}[] & \rdelim\}{3}{10pt}[$i$] & \rdelim\}{6}{10pt}[$n$] \\
& \vdots & & & \\
& \delta\left(f_{\ell i}(\sigma) \right) & & & \\
& \vdots & & & \\
& 0 & & & \\
& \vdots & & & \rdelim\}{2}{10pt}[$h_{\ell}$] \\
& 0 & & & \\
\end{array} \hspace{10pt} \in G_{\ell}\hookrightarrow \oplus_{\ell=1}^{T}G_{\ell}.
\]
Then by the definition of $\pi$ we have
\[\pi\left( \Delta(\omega)\right)=
\begin{array}{rclll}
\ldelim({8}{4pt}[] & 0 & \rdelim){8}{4pt}[] & \rdelim\}{3}{10pt}[$i$] & \rdelim\}{6}{10pt}[$n$] \\
& \vdots & & & \\
& \delta\left(f_{\ell i}(\sigma) \right) & & & \\
& \vdots & & & \\
& 0 & & & \\
& \vdots & & & \rdelim\}{2}{10pt}[$h_{s+1} + \cdots + h_{T}$] \\
& 0 & & & \\
\end{array} \hspace{50pt} \in G.
\]
On the other hand, we recall that the image of $\alpha_{\ell i}$ under the projection map $\mu :\oplus_{\ell=1}^{T}\cM_{\ell}' \twoheadrightarrow \cM$ is denoted by $\alpha_{i}$. As $\mu$ is a left $\ok[t,\sigma]$-module homomorphism, we have
\[\mu\left(f_{\ell i}(\sigma)\alpha_{\ell i} \right)=f_{\ell i}(\sigma) \alpha_{i} .\]
Recall further that $\left(\alpha_{1},\ldots,\alpha_{n},\beta_{s+1},\ldots,\beta_{T} \right)$ is a $\ok[\sigma]$-basis for $\cM$. Hence by the definition of $\Delta$, we see that 
\[ \Delta\left( \mu\left(\omega \right) \right)=\Delta\left(f_{\ell i}(\sigma) \alpha_{i} \right) \]
is equal to
\[\pi\left( \Delta(\omega)\right)=
\begin{array}{rclll}
\ldelim({8}{4pt}[] & 0 & \rdelim){8}{4pt}[] & \rdelim\}{3}{10pt}[$i$] & \rdelim\}{6}{10pt}[$n$] \\
& \vdots & & & \\
& \delta\left(f_{\ell i}(\sigma) \right) & & & \\
& \vdots & & & \\
& 0 & & & \\
& \vdots & & & \rdelim\}{2}{10pt}[$h_{s+1} + \cdots + h_{T}$] \\
& 0 & & & \\
\end{array} \hspace{50pt} \in G.
\]

Now we consider the case of $\omega=g_{\ell i}(\sigma) \beta_{\ell i}$  in (\ref{E:case2}).  We denote by
\[\bz_{\ell}:=\Delta(\omega)=
\begin{array}{rclll}
\ldelim({8}{4pt}[] & 0 & \rdelim){8}{4pt}[] & \rdelim\}{2}{10pt}[$n$] & \\
& \vdots & & & \\
& 0 & & \rdelim\}{3}{10pt}[$i$] & \rdelim\}{5}{10pt}[$h_{\ell}$] \\
& \vdots & & & \\
& \delta\left(g_{\ell i}(\sigma) \right) & & & \\
& \vdots & & & \\
& 0 & & & \\
\end{array} \hspace{10pt} \in G_{\ell}\hookrightarrow \oplus_{\ell=1}^{T}G_{\ell}.
\]
Since $s+1 \leq \ell \leq T$ and $1\leq i \leq h_{\ell}$, by the definition of $\pi$ we have
\[\pi\left( \Delta(\omega)\right)=
\begin{array}{rcll}
\ldelim({8}{4pt}[] & 0 & \rdelim){8}{4pt}[] & \rdelim\}{2}{10pt}[$n$] \\
& \vdots & & \\
& 0 & & \rdelim\}{2}{10pt}[$h_{s+1}+\cdots+h_{\ell-1}$] \\
& \vdots & & \\
& {\bz_{\ell}}_{-} & & \\
& \vdots & & \rdelim\}{2}{10pt}[$h_{\ell+1}+\cdots+h_{T}$] \\
& 0 & & \\
\end{array} \hspace{90pt} \in G.
\]

Recall that we identify $\mu(\beta_{\ell})$ with $\beta_{\ell}$. Since $\mu$ is a left $\ok[t,\sigma]$-module homomorphism, we have
\[\mu(\omega)=g_{\ell i}(\sigma)\mu( \beta_{\ell i})=g_{\ell i}(\sigma)\beta_{\ell i}. \] 
Since $(\alpha_{1},\ldots,\alpha_{n},\beta_{s+1},\ldots,\beta_{T})$ is a $\ok[\sigma]$-basis for $\cM$,  via this basis we see that $\Delta (\mu(\omega))$ is the same as 
$\pi \left( \Delta(\omega) \right)$.

Since the map $\mu$ induces an $\FF_{q}[t]$-module homomorphism
\[\bigoplus_{\ell=1}^{T}\left(\cM'_{\ell} / (\sigma-1)\cM'_{\ell} \right) \rightarrow \cM / (\sigma-1)\cM ,\] the diagram above shows that $\pi:\oplus_{\ell=1}^{T}G_{\ell}(\ok)\rightarrow G(\ok)$ is a left $\FF_{q}[t]$-module homomorphism, and hence $\pi$ is a morphism of $t$-modules since the group of $\ok$-valued points is Zariski dense inside the algebraic group in question. 

\end{proof}

\begin{corollary}\label{Cor: defined over A}
Let notation be given as above. Let $\rho_{\ell}$ be the map defining the $\FF_{q}[t]$-module structure of $G_{\ell}$ for $1\leq \ell \leq T$, and $\rho$ be the map defining the $\FF_{q}[t]$-module of $G$. If  $\rho_{\ell t}\in \Mat_{\dim G_{\ell}}(A[\tau])$ for every $1\leq \ell \leq T$, then $\rho_{t} \in  \Mat_{\dim G}(A[\tau])$.
\end{corollary}
\begin{proof}
It is clear to see that the map $\pi$ is surjective. Since $\pi$ is $\FF_{q}[t]$-linear, the result is derived from Definition~\ref{Def:pi}.
\end{proof}

\subsection{The key lemma} In this section, we give a formula which is a crucial step in the proof of Theorem~\ref{Thm: Introduction}. However, we state and prove the formulation in the setting as general as possible. We follow Brownawell and Papanikolas to introduce the notion of tractable coordinates, to which Yu's sub-t-module theorem is most easily applied.

\begin{definition}\label{Def:Tractable}
Let $L$ be a field extension over $k$ and suppose that $L^{d}:=\Mat_{d\times 1}(L)$ has a left $\FF_{q}[t]$-module structure via an $\FF_{q}$-linear ring homomorphism
\[\FF_{q}[t]\rightarrow \End_{\FF_{q}}(L^{d}) .\] The $i$th coordinate of $L^{d}$ is called tractable if the $i$th coordinate of $a \cdot \bz$ is equal to $a(\theta)z_{i}$ for any $a\in \FF_{q}[t]$ and any $\bz = (z_{1}, \dots, z_{d})^{\tr} \in L^{d}$.

Suppose that the affine variety $\mathbb{A}_{/L}^{d}$ has a left $\FF_{q}[t]$-module structure in the sense that for every field extension $L'/L$, $\mathbb{A}^{d}(L')$ has a left $\FF_{q}[t]$-module structure that is functorial in $L'$. We say that the $i$th coordinate of $\mathbb{A}_{/L}^{d}$ is tractable if for every field extension $L'/L$, the $i$th coordinate of $\mathbb{A}^{d}(L')$ is tractable.
\end{definition}
 Typical examples of tractable coordinates arise from the Lie algebras of tensor powers of the Carlitz module.  For any positive integer $s$, we note that $\Lie \bC^{\otimes s}(L)\cong L^{s}$ has a left $\FF_{q}[t]$-module structure via $\partial [-]_{s}$ for a field extension $L/k$. From (\ref{E: Def of C otimes s}) we see that the $s$th coordinate of $\Lie \bC^{\otimes s}(L)$ is tractable.

The main result in this section is the following lemma.
\begin{lemma}\label{Lemma: Key Lemma}
Let  $\cN$, $\left\{ \cM_{\ell}' \right\}_{\ell=1}^{T}$ and $\cM$ be the Anderson dual $t$-motives with hypothesis given in Sec.~\ref{Subsec: Fiber Coproduct}. Let $H$ be the $n$-dimensional $t$-module associated to $\cN$, $G_{\ell}$ be the $t$-module associated to $\cM_{\ell}'$ for $\ell=1,\ldots,T$, and $G$ be the $t$-module associated to $\cM$.  Suppose that the $n$th coordinate of $\Lie G_{\ell}(\CC_{\infty})$ is tractable for all $1\leq \ell \leq T$.  Let $Z_{\ell}\in \Lie G_{\ell} (\CC_{\infty})$ be a vector with $n$th coordinate denoted by $\cL_{\ell n}$. Let $\pi:\oplus_{\ell=1}^{T}G_{\ell}\rightarrow G$ be the morphism of $t$-modules given in Definition~\ref{Def:pi}. For each $1\leq \ell \leq T$,  let $b_{\ell}\in \FF_{q}[\theta]$ be  any polynomial and put  $\bv_{\ell}:= \exp_{G_{\ell}}(Z_{\ell})\in G_{\ell}(\CC_{\infty})$, $Z:=\partial \pi \left( (\partial[b_{\ell}(t)] Z_{\ell})_{\ell}\right)\in \Lie G(\CC_{\infty}) $ and $\bv:=\pi \left( ([b_{\ell}(t)]\bv_{\ell})_{\ell}\right)\in G(\CC_{\infty})$. Then we have
\begin{itemize}
\item[(a)] The $n$th coordinate of $Z$ is equal to $\sum_{\ell=1}^{T}b_{\ell} \cL_{\ell n}$.
\item[(b)] $\exp_{G}(Z)$ is equal to $\bv$.
\end{itemize}
\end{lemma}

\begin{proof}
By the canonical identification $\Lie\left(\oplus_{\ell}G_{\ell}(\CC_{\infty})\right)=\oplus_{\ell}\Lie G_{\ell}(\CC_{\infty})$,  we have the following commutative diagram  according to (\ref{E:ExpFunctorial}):
\begin{equation}\label{E:diagram for pi and exp}
 \xymatrix{
\oplus_{\ell=1}^{T}G_{\ell}(\CC_{\infty})  \ar[r]^{\pi} & G(\CC_{\infty})\\
\oplus_{\ell=1}^{T}\Lie G_{\ell}(\CC_{\infty}) \ar[u]^{\oplus_{\ell} \exp_{G_{\ell}}} \ar[r]^{\partial \pi}          &\Lie G(\CC_{\infty})\ar[u]^{\exp_{G}} .}
 \end{equation} 
 Property $(b)$ follows from the diagram above.
 
 To prove $(a)$, we note that the $n$th coordinate of $\partial [b_{\ell}(t)]Z_{\ell}$ is given by $b_{\ell} \cL_{\ell n}$ since by hypothesis the $n$th coordinate of $\Lie G_{\ell}(\CC_{\infty})$ is tractable. By Definition~\ref{Def:pi} the morphism $\pi$ has no $\tau$-terms when expressing it as a matrix with entries in $\ok[\tau]$. So the induced morphism $\partial \pi$  has the same form as $\pi$ (see~(\ref{E:partial rho a})), implying the desired property from the definition of $Z$.
\end{proof}

\begin{remark} If we take $\cN$ and all $\left\{ \cM_{\ell}'\right\}_{\ell=1}^{T}$ to be $C^{\otimes n}$, then the fiber coproduct of $\left\{ \cM_{\ell}'\right\}_{\ell=1}^{T}$ over $\cN$ is $C^{\otimes n}$ and hence its associated $t$-module $G$ is $\bC^{\otimes n}$. In this case, the morphism $\pi:\oplus_{\ell=1}^{T}\bC^{\otimes n}\rightarrow \bC^{\otimes n}$ is the sum of vectors. This special case  would help the reader understand how one uses Lemma~\ref{Lemma: Key Lemma} to generalize \cite[Thm.~3.8.3(I)]{AT90} to higher depth MZV's in Sec.~\ref{Sec: proof of main thm}

\end{remark}

\section{The convergence of $\log_{G_{\ell}}(\bv_{\ell})$}\label{Sec: log G ell}
In this section, we consider the $t$-module and special point constructed in~\cite{CM17}, and the primary goal is to show Theorem~ \ref{theorem_log_li} asserting that the logarithm of the $t$-module in question converges $\infty$-adically at the special point, and certain coordinates of the logarithm give Carlitz multiple star polylogarithms. To prove Theorem~\ref{Thm: Introduction}, the results presented in this section are applied in Section~\ref{Sec: Proof of Main Thm} to illustrate that all the conditions of Lemma~\ref{Lemma: Key Lemma} are satisfied for our setting.

\subsection{The constructions of the $t$-module and special point}\label{Subsec: Gu}

In what follows, we fix $\fs = (s_{1}, \dots, s_{r})\in \NN^{r}$ and $\bu = (u_{1}, \dots, u_{r}) \in (\ok^{\times})^{r}$.  We will define a pair $(G,\bv)$ associated to $\fs$ and $\bu$, where $G$ is a $t$-module defined over $\ok$ and $\bv\in G(\ok)$.

Put $L_{0}:=1$, and $L_{i}:=(\theta-\theta^{q})\cdots (\theta-\theta^{q^{i}})$ for $i\in\NN$. We define the $\fs$th Carlitz multiple polylogarithm, abbreviated as CMPL, as follows (see~\cite{C14}):
\begin{equation}\label{E:Def of CMPL}
\Li_{\fs}(z_1,\ldots,z_r):=\sum_{i_1>\cdots >i_r \geq 0} \frac{z_{1}^{q^{i_{1}}}\cdots z_{r}^{q^{i_{r}}}  }  {L_{i_1}^{s_{1}}\cdots L_{i_r}^{s_{r}}} . 
\end{equation}
To avoid heavy notation on the subscript, we use the same notation $\Li_{\fs}$ in the function field setting. Since we no longer use the classical multiple polylogarithms in the later context, there will not be misunderstanding.

We also define the $\fs$th Carlitz multiple star polylogarithm, abbreviated as CMSPL, as follows (see \cite{CM17}):
 \begin{equation}\label{E:Def of CMSPL}
\Li_{\fs}^{\star}(z_1,\ldots,z_r):=\sum_{i_1\geq \cdots  \geq i_r \geq 0} \frac{z_{1}^{q^{i_{1}}}\cdots z_{r}^{q^{i_{r}}}  }  {L_{i_1}^{s_{1}}\cdots L_{i_r}^{s_{r}}} . 
\end{equation}

\begin{remark}\label{Rem:NonvanDomainLi}
For an $r$-tuple $\fs=(s_{1},\ldots,s_{r})\in \NN^{r}$, we put
\[\mathbb{D}_{\fs}':=\left\{(x_{1},\ldots,x_{r})\in \CC_{\infty}^{r}: |x_{i}|_{\infty}< q^{\frac{s_{i}q}{q-1}} \hbox{ for }i=1,\ldots,r \right\}\subset \mathbb{D}_{\fs}''  ,\] where

\begin{equation}\label{Def: Ds''}
\mathbb{D}_{\fs}'':=\left\{(x_{1},\ldots,x_{r})\in \CC_{\infty}^{r}: |x_{1}|_{\infty} < q^{\frac{s_{1}q}{q-1}} \hbox{ and } |x_{i}|_{\infty} \leq q^{\frac{s_{i}q}{q-1}} \hbox{ for }i=2,\ldots,r \right\}  .
\end{equation}

Since $\Li_{\fs}(\bx)$ and $\Li_{\fs}^{\star}(\bx)$ have the same general terms, by \cite[Rem.~5.1.5]{C14} these two series converge $\infty$-adically for any $\bx \in \mathbb{D}_{\fs}''$, and $\Li_{\fs}(\bx)$ is non-vanishing for any $\bx\in \mathbb{D}_{\fs}'\cap (\CC_{\infty}^{\times })^{r}$. We mention that $\mathbb{D}_{\fs}''$ is used in Theorem~\ref{theorem_log_li}.
\end{remark}

For $1 \leq \ell \leq r$,
we put $d_{\ell} := s_{\ell} + \cdots + s_{r}$ and $d := d_{1} + \cdots + d_{r}$.
Let $B$ be a $d \times d$-matrix of the form

\[
\left( \begin{array}{c|c|c}
B[11] & \cdots & B[1r] \\ \hline
\vdots & & \vdots \\ \hline
B[r1] & \cdots & B[rr]
\end{array} \right),
\]
where $B[\ell m]$ is a $d_{\ell} \times d_{m}$-matrix for each $\ell$ and $m$ and we call $B[\ell m]$ the $(\ell, m)$-th block sub-matrix of $B$.

For $1 \leq \ell \leq m \leq r$, we define the following matrices:

\[
N_{\ell} := \left(
\begin{array}{ccccc}
0 & 1 & 0 & \cdots & 0 \\
& 0 & 1 & \ddots & \vdots \\
& & \ddots & \ddots & 0 \\
& & & \ddots & 1 \\
& & & & 0
\end{array}
\right)
\in \Mat_{d_{\ell}}(\ok),
\]

\[
N := \left(
\begin{array}{cccc}
N_{1} & & & \\
& N_{2} & & \\
& & \ddots & \\
& & & N_{r}
\end{array}
\right)
\in \Mat_{d}(\ok),
\]

\[
E[\ell m] := \left(
\begin{array}{cccc}
0 & \cdots & \cdots & 0 \\
\vdots & \ddots & & \vdots \\
0 & & \ddots & \vdots \\
1 & 0 & \cdots & 0
\end{array}
\right)
\in \Mat_{d_{\ell} \times d_{m}}(\ok) \ \ \ (\mathrm{if} \ \ell = m),
\]

\[
E[\ell m] := \left(
\begin{array}{cccc}
0 & \cdots & \cdots & 0 \\
\vdots & \ddots & & \vdots \\
0 & & \ddots & \vdots \\
(-1)^{m-\ell} \prod_{e=\ell}^{m-1} u_{e} & 0 & \cdots & 0
\end{array}
\right)
\in \Mat_{d_{\ell} \times d_{m}}(\ok) \ \ \ (\mathrm{if} \ \ell < m),
\]

\[
E := \left(
\begin{array}{cccc}
E[11] & E[12] & \cdots & E[1r] \\
& E[22] & \ddots & \vdots \\
& & \ddots & E[r-1,r] \\
& & & E[rr]
\end{array}
\right)
\in \Mat_{d}(\ok).
\]
We further define

\[
E_{m} := \left( \begin{array}{c|c|c}
0 & 0 & 0 \\ \hline
0 & E[mm] & 0 \\ \hline
0 & 0 & 0
\end{array} \right) \in \Mat_{d}(\ok)
\]
to be the $d \times d$-matrix such that the $(m,m)$-th block sub-matrix is $E[mm]$
and the others are zero matrices.

We then define the $t$-module $G = G_{\fs, \bu} := (\GG_{a}^{d}, \rho)$ by
\begin{equation}\label{E:Explicit t-moduleCMPL}
  \rho_{t} = \theta I_{d} + N + E \tau
  \in \Mat_{d}(\ok[\tau]),
\end{equation}
and note that $G$ depends  only on $u_{1},\ldots,u_{r-1}$. Finally, we define the special point
\begin{equation}\label{E:v_s,u}
\bv: = \bv_{\fs, \bu} :=
\begin{array}{rcll}
\ldelim( {15}{4pt}[] & 0 & \rdelim) {15}{4pt}[] & \rdelim\}{4}{10pt}[$d_{1}$] \\
& \vdots & & \\
& 0 & & \\
& (-1)^{r-1} u_{1} \cdots u_{r} & & \\
& 0 & & \rdelim\}{4}{10pt}[$d_{2}$] \\
& \vdots & & \\
& 0 & & \\
& (-1)^{r-2} u_{2} \cdots u_{r} & & \\
& \vdots & & \vdots \\
& 0 & & \rdelim\}{4}{10pt}[$d_{r}$] \\
& \vdots & & \\
& 0 & & \\
& u_{r} & & \\[10pt]
\end{array} \in G(\ok).
\end{equation}

\begin{remark}\label{Rem:Integral}
If $\bu\in A^{r}$, then $\rho_{t}\in \Mat_{d}(A[\tau])$ and $\bv\in G(A)$. 
\end{remark}

\begin{remark}\label{Rem: G and M}
The $t$-module $G$ above is the $t$-module associated to the Anderson dual $t$-motive $\cM'$, where $\cM'$ is free of rank $r$ over $\ok[t]$ and the representing matrix by $\sigma$ on certain $\ok[t]$-basis for $\cM'$ is given by
\begin{equation}\label{E:Phi'}
\Phi' :=
\begin{pmatrix}
                (t-\theta)^{s_{1}+\cdots+s_{r}}  &  &  &  \\
                u_{1}^{(-1)}(t-\theta)^{s_{1}+\cdots+s_{r}}  & (t-\theta)^{s_{2}+\cdots+s_{r}}   &  &  \\
                  & \ddots & \ddots &  \\
                  &  & u_{r-1}^{(-1)}(t-\theta)^{s_{r-1}+s_{r}} & (t-\theta)^{s_{r}}  \\
\end{pmatrix}\in \Mat_{r}(\ok[t]),
\end{equation}
where \[\left\{  (t-\theta)^{s_{1}+\cdots+s_{r}},(t-\theta)^{s_{2}+\cdots+s_{r}} , \ldots,  (t-\theta)^{s_{r}} \right\}\]  are the diagonals and \[\left\{ u_{1}^{(-1)}(t-\theta)^{s_{1}+\cdots+s_{r}}, \ldots,  u_{r-1}^{(-1)}(t-\theta)^{s_{r-1}+s_{r}}\right\}\] are displayed below the diagonals. We note that $\cM'$ is an iterated extension of  tensor powers of the Carlitz $t$-motive.
\end{remark}

\begin{remark}\label{Rem: R.A.T of M'}
The Anderson dual $t$-motive $\cM'$ contains $C^{\otimes n}$ as a {\it{saturated}} sub-Anderson dual $t$-motive (see~\cite[Sec.~4.3.3]{ABP04}). Moreover,  $\cM'$ is rigid analytically trivial since a rigid analytic trivialization $\Psi'\in \GL_{r}(\TT)$ is given as the upper left square of $\Psi$ given in \cite[(2.3.7)]{CPY14} by changing $(Q_{1},\ldots,Q_{r})$ to $(u_{1},\ldots,u_{r})$.

\end{remark}

\subsection{The convergence}

To study the $\infty$-adic convergence issue about $\log_{G}$ at $\bv$, we adopt some techniques of \cite[2.4.3]{AT90}. We denote by
\[
\log_{G} = \sum_{i \geq 0} P_{i} \tau^{i}
\]
the logarithm of the $t$-module $G$, where $P_{0} = I_{d}$ and $P_{i} \in \Mat_{d}(\ok)$ for all $i$.

For a matrix $\gamma:=(\gamma_{ij})$ with entries in $\CC_{\infty}$, we put
\[|\gamma|_{\infty}:=\max_{i,j}\left\{ |\gamma_{ij}|_{\infty}\right\} .\]

\begin{lemma} \label{lemma_abs_P} Let $\fs=(s_{1},\ldots,s_{r})\in \NN^{r}$ and $\bu=(u_{1},\ldots,u_{r})\in (\ok^{\times })^{r}$. 
If $|u_{\ell}|_{\infty} \leq q^{\frac{s_{\ell} q}{q-1}}$ for each $1 \leq \ell < r$,
then we have
\[
|P_{i} N^{d_{\ell}-j} E_{\ell}|_{\infty} \leq q^{(d_{\ell}-j) q^{i} - (d_{\ell} q^{i} - d_{1}) \frac{q}{q-1}}
\]
for each $i, j, \ell$ with $i \geq 0$,  $1 \leq \ell \leq r$, and $1 \leq j \leq d_{\ell}$.
\end{lemma}

\begin{proof}
Note that the $(d_{1} + \cdots + d_{\ell-1} + 1)$th column of $P_{i} N^{d_{\ell}-j} E_{\ell}$ is the $(d_{1} + \cdots + d_{\ell-1} + j)$th column of $P_{i}$, and the other columns are zero vectors.
When $i = 0$, the inequality holds clearly.
Let $i \geq 1$ and assume that the inequality holds for $i$.
By \cite[3.2.4]{CM17}, we have
\[
P_{i+1} N^{d_{\ell}-j} E_{\ell}
= - \sum_{m=0}^{2d_{1}-2} \dfrac{1}{(\theta^{q^{i+1}}-\theta)^{m+1}} \sum_{n=0}^{m} (-1)^{n} \binom{m}{n} N^{m-n} P_{i} E^{(i)} N^{n+d_{\ell}-j} E_{\ell}.
\]
Note that $E^{(i)} N^{n+d_{\ell}-j} E_{\ell} = 0$ for $n \neq j-1$, and $N^{m-n} = 0$ for $m-n \geq d_{1}$.
Thus we have
\begin{eqnarray*}
P_{i+1} N^{d_{\ell}-j} E_{\ell}
& = & \sum_{m=j-1}^{d_{1}+j-2} \dfrac{(-1)^{j}}{(\theta^{q^{i+1}}-\theta)^{m+1}} \binom{m}{j-1} N^{m-j+1} P_{i} E^{(i)} N^{d_{\ell}-1} E_{\ell} \\
& = & \sum_{m=j-1}^{d_{1}+j-2} \dfrac{(-1)^{j}}{(\theta^{q^{i+1}}-\theta)^{m+1}} \binom{m}{j-1} N^{m-j+1} \sum_{n=1}^{\ell}  (-1)^{\ell-n} P'_{i,\ell,n} \prod_{n \leq e \leq \ell-1} u_{e}^{q^{i}},
\end{eqnarray*}
where $P'_{i,\ell,n}$ is the matrix such that the $(d_{1} + \cdots + d_{\ell-1} + 1)$th column is the $(d_{1} + \cdots + d_{n-1} + d_{n})$th column of $P_{i}$, and the other columns are zero vectors.

By the induction hypothesis, we obtain
\begin{eqnarray*}
\Bigl|P'_{i,\ell,n}  \prod_{n \leq e \leq \ell-1} u_{e}^{q^{i}}  \Bigr|_{\infty}
& \leq & q^{(d_{n}-d_{n}) q^{i} - (d_{n}q^{i}-d_{1}) \frac{q}{q-1}} \cdot  \prod_{n \leq e \leq \ell-1} q^{\frac{s_{e} q}{q-1} \cdot q^{i}}  \\
& = & q^{- (d_{n}q^{i}-d_{1}) \frac{q}{q-1}} \cdot q^{(d_{n}-d_{\ell}) \frac{q^{i+1}}{q-1}} \\
& = & q^{- (d_{\ell}q^{i}-d_{1}) \frac{q}{q-1}}.
\end{eqnarray*}
Therefore we have
\begin{eqnarray*}
|P_{i+1} N^{d_{\ell}-j} E_{\ell}|_{\infty}
& \leq & \max_{j-1 \leq m \leq d_{1}+j-2} \left\{ q^{-(m+1)q^{i+1}} \right\} \cdot q^{- (d_{\ell}q^{i}-d_{1}) \frac{q}{q-1}} \\
& = & q^{- j q^{i+1}} \cdot q^{- (d_{\ell}q^{i}-d_{1}) \frac{q}{q-1}} \\
& = & q^{(d_{\ell}-j) q^{i+1} - (d_{\ell}q^{i+1}-d_{1}) \frac{q}{q-1}}.
\end{eqnarray*}
\end{proof}

\begin{proposition} \label{proposition_conv_log}
Assume $|u_{\ell}|_{\infty} \leq q^{\frac{s_{\ell} q}{q-1}}$ for each $1 \leq \ell < r$.
Take a point $\bx = (x_{m}) \in G(\CC_{\infty})$ such that
\[
|x_{d_{1} + \cdots + d_{\ell-1} + j}|_{\infty} < q^{- (d_{\ell}-j) + \frac{d_{\ell} q}{q-1}}
\]
for each $j, \ell$ with $1 \leq \ell \leq r$ and $1 \leq j \leq d_{\ell}$.
Then $\log_{G}(\bx)$ converges in $\Lie G(\CC_{\infty})$.
\end{proposition}

\begin{proof}
By Lemma \ref{lemma_abs_P}, we have
\begin{eqnarray*}
|P_{i} \bx^{(i)}|_{\infty}
& \leq & \max_{j, \ell} \left\{ q^{(d_{\ell}-j) q^{i} - (d_{\ell}q^{i}-d_{1}) \frac{q}{q-1}} \cdot |x_{d_{1} + \cdots + d_{\ell-1} + j}|_{\infty}^{q^{i}} \right\} \\
& = & \max_{j, \ell} \left\{ q^{\frac{d_{1} q}{q-1}} \cdot \left( |x_{d_{1} + \cdots + d_{\ell-1} + j}|_{\infty} \Big/ q^{-(d_{\ell}-j) + \frac{d_{\ell} q}{q-1}} \right)^{q^{i}} \right\} \\
& \to & 0 \ \ (i \to \infty).
\end{eqnarray*}
\end{proof}

\begin{theorem} \label{theorem_log_li}
Given any $\fs=(s_{1},\ldots,s_{r})\in \NN^{r}$, we put $\widetilde{\fs}:=(s_{r},\ldots,s_{1})$ and let $\mathbb{D}_{\widetilde{\fs}}''$ be defined in (\ref{Def: Ds''}). Suppose that we have $\bu=(u_{1},\ldots,u_{r})\in (\ok^{\times })^{r}$ for which $\widetilde{\bu}:=(u_{r},\ldots,u_{1})\in \mathbb{D}_{\widetilde{\fs}}''$.  Let $G$ and $\bv$ be defined as above associated to $\fs$ and $\bu$. Then $\log_{G}$ converges $\infty$-adically at $\bv$ and we have the formula
\[
\log_{G} (\bv) =
\begin{array}{rcll}
\ldelim( {15}{4pt}[] & * & \rdelim) {15}{4pt}[] & \rdelim\}{4}{10pt}[$d_{1}$] \\
& \vdots & & \\
& * & & \\
& (-1)^{r-1}\Li_{(s_{r}, \dots, s_{1})}^{\star}(u_{r}, \dots, u_{1}) & & \\
& * & & \rdelim\}{4}{10pt}[$d_{2}$] \\
& \vdots & & \\
& * & & \\
& (-1)^{r-2}\Li_{(s_{r}, \dots, s_{2})}^{\star}(u_{r}, \dots, u_{2}) & & \\
& \vdots & & \vdots \\
& * & & \rdelim\}{4}{10pt}[$d_{r}$] \\
& \vdots & & \\
& * & & \\
& \Li^{\star}_{s_{r}}(u_{r}) & & \\[10pt]
\end{array} \ \ \ \in \Lie G(\CC_{\infty}).
\] In particular, the $(s_{1}+\cdots+s_{r})$th coordinate of $\log_{G}(\bv)$ is $(-1)^{\dep(\fs)-1} \Lis_{\widetilde{\fs}}(\widetilde{\bu})$.

\end{theorem}

\begin{proof}
For each $1 \leq \ell \leq r$, the $(d_{1} + \cdots + d_{\ell-1} + d_{\ell})$th component of $\bv$ is $(-1)^{r-\ell} u_{\ell} u_{\ell+1} \cdots u_{r}$, and we have
\[
|(-1)^{r-\ell} u_{\ell} u_{\ell+1} \cdots u_{r}|_{\infty}
< q^{\frac{s_{\ell} q}{q-1}} \cdot q^{\frac{s_{\ell+1} q}{q-1}} \cdot \cdots \cdot q^{\frac{s_{r} q}{q-1}}
= q^{\frac{d_{\ell} q}{q-1}}
= q^{- (d_{\ell} - d_{\ell}) + \frac{d_{\ell} q}{q-1}}.
\]
Thus $\log_{G}(\bv)$ converges $\infty$-adically by Proposition \ref{proposition_conv_log}.

Arguments  proving the second assertion are entirely the same as the calculations in the proof of \cite[3.3.3]{CM17}, where we just change the $v$-adic convergence to $\infty$-adic convergence. 
\end{proof}

\begin{remark}
 We mention that all other coordinates of $\log_{G} (\bv)$ can be explicitly written down in terms of Taylor coefficients of $t$-motivic CMSPL's in~\cite{CGM19}.
\end{remark}

\section{Proof of Theorem~\ref{Thm: Introduction}}\label{Sec: Proof of Main Thm}

\subsection{Logarithmic interpretation and formulae for MZV's via CMPL's}\label{Sec: Formulae for MZV's via CMPL's}

When $r=1$ and $\fs=1\in \NN$, the series (\ref{E:Def of CMPL}) is called the Carlitz logarithm, which is the formal inverse of the exponential map of the Carlitz module $\bC$. For $r=1$ and any $s\in \NN$, the series $\Li_{s}$ in  (\ref{E:Def of CMPL})   is called the $s$th Carlitz polylogarithm studied in \cite{AT90}.  Unlike the classical case where there is a simple identity between $\zeta(\fs)$ and a particular specialization of a classical multiple polylogarithm, $\zeta_{A}(\fs)$ is in fact a $k$-linear combination of $\Li_{\fs}$ at some integral points, which will be reviewed in the following section. It turns out that such an identity for $\zeta_{A}(\fs)$ is a crucial connection that enables us to give a logarithmic interpretation for $\zeta_{A}(\fs)$ in Theorem~\ref{Thm: Introduction}.

\begin{remark}
For recent advances of transcendence theory for CMPL's, see~\cite{CY07, M17}.
\end{remark}

We now recall the Carlitz factorials. We set $D_{0}:=1$, and $D_{i}:=\prod_{j=0}^{i-1}(\theta^{q^{i}} -\theta^{q^{j}})\in A$ for $i \in \NN$.  Given a non-negative integer $n$, we express $n$ as $n=\sum_{i\geq 0 }n_{i}q^{i}$ for $0\leq n_{i}\leq q-1$. The Carlitz factorial is defined as 
\begin{equation}\label{E:Def of Gamma}
\Gamma_{n+1}:=\prod_{i}D_{i}^{n_{i}} \in A .  
\end{equation}

The aim of this section is to prove the following theorem.
\begin{theorem}\label{Thm: Infty-adic LogMZV}
Given any $r$-tuple $\fs=(s_{1},\ldots,s_{r})\in \NN^{r}$, we put $n:=\wt(\fs)$. We explicitly construct a uniformizable $t$-module $G_{\fs}$ that is defined over $k$, a special point $\bv_{\fs} \in G_{\fs}(k)$ and a vector $Z_{\fs}\in \Lie G_{\fs}(\CC_{\infty})$ so that 
\begin{itemize}
\item[(a)] $\Gamma_{s_{1}}\cdots \Gamma_{s_{r}}\zeta_{A}(\fs)$ occurs as the $n$th coordinate of $Z_{\fs}$. 
\item[(b)] $\exp_{G_{\fs}}(Z_{\fs})=\bv_{\fs}$.
\end{itemize}
\end{theorem}

To introduce the formula of $\zeta_{A}(\fs)$ in terms of $\Li_{\fs}$, we need to review the Anderson-Thakur polynomials~\cite{AT90, AT09}.  Let $t$ be an independent variable from $\theta$. We put $F_{0}:=1$ and define polynomials $F_{i}\in A[t]$ for $i \in \NN$ by the product
\[
F_{i}=\prod_{j=1}^{i}\left( t^{q^{i}}-\theta^{q^{j}} \right).
\]  We then define  the sequence of Anderson-Thakur polynomials $H_{n}\in A[t]$ (for non-negative integers $n$) by the generating function identity
\[
\left( 1-\sum_{i=0}^{\infty} \frac{  F_{i}  }{ D_{i}|_{\theta=t}} x^{q^{i}}  \right)^{-1}=\sum_{n=0}^{\infty} \frac{H_{n}}{\Gamma_{n+1}|_{\theta=t}} x^{n},
\] and note that they satisfy the following important interpolation formula~\cite[(3.7.4)]{AT90}. For integers $d\geq 0$ and $s \geq1$,  we have
\begin{equation}\label{E: Interpolation}
\left({H_{s-1}}^{(d)}\right)\Bigr|_{t=\theta}=\Gamma_{s}\cdot S_{d}(s)\cdot L_{d}^{s},
\end{equation}
where $S_{d}(s)$ is the sum of the following reciprocal polynomials
\[ S_{d}(s):=\sum_{a\in A_{d,+}} \frac{1}{a^{s}}\in k,\]where $A_{d,+}$ is the set of monic polynomials of degree $d$ in $A$
and $L_{d}$ is defined in Sec.~\ref{Subsec: Gu}.
Define the sup-norm $\| f\|:={\rm{max}}_{i}\left\{ |a_{i}|_{\infty}\right\} $ for polynomials $f=\sum_{i}a_{i}t^{i}\in \CC_{\infty}[t]$, and note further that the Anderson-Thakur polynomials have the  following property 
\begin{equation}\label{E: Coeff bound for H_n}
\| H_{n-1}(t) \| < |\theta|_{\infty}^{\frac{nq}{q-1}}
\end{equation}
for every $n\in \NN$.

\begin{remark}
The bound above comes from \cite[(3.7.3)]{AT90}. However, we shall mention about the difference of notation. Our $H_{n}(t)$ is exactly the same as $H_{n}(y,T)$ in \cite{AT90} replacing $y$ by $\theta$ and replacing $T$ by $t$. One can compare with \cite{AT09}, where their $T$ is referred to our $t$ and their $t$ is referred to our $\theta$. 
\end{remark}

In what follows, we fix an $r$-tuple of positive integers $\fs=(s_{1},\ldots,s_{r})\in \NN^{r}$. For each $1\leq i\leq r$, we expand the Anderson-Thakur polynomial $H_{s_{i}-1}(t)\in A[t]$ as
\begin{equation}\label{E:expansionHs}
H_{s_{i}-1}(t)=\sum_{j=0}^{m_{i}} u_{ij}t^{j},
\end{equation}
where $u_{ij}\in A$ with $u_{im_{i}} \neq 0$ and by (\ref{E: Coeff bound for H_n}) it satisfies
\begin{equation}\label{E: bound for u_ij}
 |u_{ij}|_{\infty}<q^{\frac{s_{i}q}{q-1}}\hbox{ for }j=0,\ldots,m_{i}  .
\end{equation}

We define
\[ J_{\fs}:=\left\{0,1,\ldots,m_{1} \right\} \times \cdots \times \left\{0,1,\ldots,m_{r} \right\}     .\]
For each $\bj=(j_{1},\ldots,j_{r})\in J_{\fs}$, we set
\[ \bu_{\bj}:=(u_{1j_{1}},\ldots,u_{rj_{r}})\in A^{r},   \]
and
\[ a_{\bj}:=a_{\bj}(t):=t^{j_{1}+\cdots+j_{r}}    .\]
Note that by (\ref{E: bound for u_ij}), we have $\bu_{\bj}\in \mathbb{D}_{\fs}'$ for every $\bj\in J_{\fs}$.

Set $\Gamma_{\fs}:=\Gamma_{s_{1}}\cdots\Gamma_{s_{r}}\in A$. The first author of the present paper established the following formula that extends  the work of Anderson-Thakur~\cite{AT90} for $r=1$.
\begin{theorem} \label{T:Chang} {\rm{(\cite[Thm.~5.5.2]{C14})}}
For each $\fs=(s_{1},\ldots,s_{r})\in \NN^{r}$, let $J_{\fs}$, $a_{\bj}$ and $\bu_{\bj}$ be defined as above. Then the following identity holds.
\[\Gamma_{\fs} \zeta_{A}(\fs)=\sum_{\bj \in J_{\fs}}a_{\bj}(\theta)\Li_{\fs}(\bu_{\bj}).  \]
\end{theorem}

\begin{remark}
The identity above is based on the interpolation formula~\eqref{E: Interpolation}. Thakur~\cite{T92} initiated the study of generalizing the power sum $S_{d}(s)$ to general ${\bf{A}}$, which is the ring of regular functions on a smooth, projective and geometrically irreducible curve over $\FF_{q}$ regular away from a fixed closed point. We mention that a good analogue of the Anderson-Thakur polynomials for ${\bf{A}}$ satisfying an interpolation formula involving power sums such as~\eqref{E: Interpolation} will be very helpful for the study of MZV's over ${\bf{A}}$, but so far such an analogue is not known yet.
\end{remark}

In the following section, we have to express the right hand side of the identity in Theorem~\ref{T:Chang} in terms of CMSPL's since such a formulation plays a crucial role in the proof of Theorem~\ref{Thm: Introduction}. 

\subsection{Formulae for MZV's via CMSPL's} To express CMPL in terms of CMSPL's, one just needs the inclusion-exclusion principle on the set 
\[\left\{ i_{1}>\cdots>i_{r}\geq 0\right\} .\]
We take a simple example for $r=2$, which would simply allow one to understand what we do in the more general setting. Since we have
\[ \left\{i_{1}>i_{2}\geq 0 \right\}=\left\{ i_{1}\geq i_{2} \geq 0 \right\} \backslash \left\{ i_{1}=i_{2}\geq 0 \right\},  \]
it follows that
\[\Li_{(s_{1},s_{2})}(z_{1},z_{2}) =\Li_{(s_{1},s_{2})}^{\star}(z_{1},z_{2})-\Li_{s_{1}+s_{2}} ^{\star}(z_{1}+z_{2}) .\]
So one can obtain that $\Li_{\fs}$ can be expressed as a linear combination of CMSPL's, which is presented in Proposition~\ref{P:Li as Lis}. 

In what follows, the main target is to express $\Gamma_{\fs} \zeta_{A}(\fs)$ explicitly as an $A$-linear combination of some CMSPL's at certain integral points. The details of the procedure below are to explain that $\fs_{\ell}$, $b_{\ell}$ and $\bu_{\ell}$ in Theorem~\ref{T:MZV as Lis} can be written down explicitly, and that will explain why the constructions of $G_{\fs}$ and $\bv_{\fs}$ in Theorem~\ref{Thm: Introduction} are explicit. So we suggest the reader to skip  ahead to Theorem~\ref{T:MZV as Lis}  unless one needs explicit examples of $G_{\fs}$ and $\bv_{\fs}$.

 An index of depth $r$ is defined to be an $r$-tuple $\fs=(s_{1},\ldots,s_{r})\in \NN^{r}$.
\begin{definition}
Let $\fs=(s_{1},\ldots,s_{r})\in\NN^{r}$ be an index of depth $r>1$. Let $S$ be the set consisting of the two symbols '$,$' (comma) and '$+$' (addition) and $S^{\times}$ be the set consisting of the two symbols \lq$,$\rq \ and \lq $\times$\rq \ (multiplication).

\begin{enumerate}
\item We define a map $\lambda:=\left(\bw\mapsto \bw^{\times} \right):S^{r-1}\rightarrow {S^{\times}}^{r-1} $ by leaving \lq$,$\rq \ be fixed and changing \lq$+$\rq \ to \lq$\times$\rq. That is, if $\bw=(\bw_{1},\ldots,\bw_{r-1})$, then $\bw_{i}^{\times}:=$ \lq$,$\rq  \ if $\bw_{i}=$ \lq$,$\rq;  otherwise $\bw_{i}^{\times}:=$ \lq$\times$\rq.

\item For  any  $\bw=(\bw_{1},\ldots,\bw_{r-1})\in S^{r-1}$, we define $\bw(\fs):=(s_{1}\bw_{1}s_{2}\bw_{2}\cdots \bw_{r-1}s_{r})$. That is, $\bw(\fs)$ is a tuple of positive integers obtained from $(s_{1},\ldots,s_{r})$ by inserting the symbol $\bw_{i}$ between $s_{i}$ and $s_{i+1}$ for $i=1,\ldots,r-1$.

 \item For  any  $\bw=(\bw_{1},\ldots,\bw_{r-1})\in S^{r-1}$ and $\bu=(u_{1},\ldots,u_{r})\in \bar{k}^{r}$, we define $\bw^{\times}(\bu):=(u_{1}\bw_{1}^{\times}u_{2}\bw_{2}^{\times}\cdots \bw_{r-1}^{\times}u_{r})$. That is, $\bw^{\times}(\bu)$ is the tuple of algebraic elements over $k$ obtained from $(u_{1},\ldots,u_{r})$ by inserting the symbol $\bw_{i}^{\times}$ between $u_{i}$ and $u_{i+1}$ for $i=1,\ldots,r-1$.
\end{enumerate}
\end{definition}

For example, let $\bw=(\bw_1,\bw_2)$ with $\bw_1=$ `$,$' and $\bw_2=$ `$+$'. Then for $\fs=(s_{1},s_{2},s_{3})\in \NN^{3}$, we have
\[ \bw(\fs)=(s_{1},s_{2}+s_{3}) .\]
Furthermore, for $\bu=(u_{1},u_{2},u_{3})\in\ok^{3}$, we have

\[ \bw^{\times}(\bu)=(u_{1},u_{2}u_{3})   .\]
Finally, we define $\nu(\bw)$ to be the number of `$+$' in $\bw$.

\begin{proposition}\label{Rem: Conv at w star}
Fix an index $\fs\in \NN^{r}$ with $r>1$. Then for any $\bu\in \mathbb{D}_{\fs}'$ (resp.\ $\bu \in \mathbb{D}_{\fs}''$) and $\bw\in S^{r-1}$, we have that $\bw^{\times}(\bu)\in \mathbb{D}_{\bw(\fs)}'$ (resp.\ $\bw^{\times}(\bu) \in \mathbb{D}_{\bw(\fs)}''$). In particular, $\Li_{\bw(\fs)}^{\star}(\bw^{\times}(\bu))$ converges by Remark~\ref{Rem:NonvanDomainLi}.
\end{proposition}
\begin{proof}
The assertion follows immediately from the non-archimedean property of $|\cdot|_{\infty}$.
\end{proof}

In order to make our formula of MZV's convenient for use, for $r=1$ we simply define $S^{r-1}=S^{0}:=\left\{ identity\right\}$, and  denote by 
$P(s)=s$, $P^{\times }(u)=u$ and $\nu(P):=0$ for $P\in S^{0}$.

Applying the inclusion-exclusion principle on the set $\left\{ i_{1}>\cdots >i_{r}\geq 0 \right\}$, we have the following identity. 

\begin{proposition}\label{P:Li as Lis}
 Let $r $ be a positive integer, $\fs\in \NN^{r}$ be an index and $z_{1},\ldots,z_{r}$ be  $r$ independent variables. Putting $\bz=(z_{1},\ldots,z_{r})$. Then the following identity holds:
\[
\Li_{\fs}(\bz) = \sum_{P\in S^{r-1}} (-1)^{\nu(P)} \Lis_{P(\fs)}(P^{\times}(\bz)).
\]

\end{proposition}

\begin{remark}
The similar statement of Proposition \ref{P:Li as Lis} for classical MZV's can be seen in \cite{Ya13}.
\end{remark}

Recall that  the special points $\bu_{\bj}$ in Theorem~\ref{T:Chang} belong to $\mathbb{D}_{\fs}'$ for every $\bj\in J_{\fs}$, and so $\Lis_{P(\fs)}(P^{\times}(\bu_{\ell}))$ converges by Proposition~\ref{Rem: Conv at w star}.  Combining Theorem~\ref{T:Chang} and the proposition above, we have the following expression for $\zeta_{A}(\fs)$ in terms of CMSPL's.

\begin{theorem}\label{T:MZV as Lis}
For any depth $r$ index $\fs\in \NN^{r}$, there are explicit tuples $\fs_{\ell}\in \NN^{\tiny{\rm{dep}(\fs_{\ell})}}$ with $\tiny{\rm{wt}(\fs_{\ell})}=\tiny{\rm{wt}(\fs)}$, $\tiny{\rm{dep}(\fs_{\ell})}\leq r$, explicit coefficients $b_{\ell}\in A$ and vectors $\bu_{\ell}\in A^{\tiny{\rm{dep}(\fs_{\ell})}}$ so that
\[\Gamma_{\fs} \zeta_{A}(\fs)=\sum_{\ell}b_{\ell} \cdot (-1)^{\tiny{\rm{dep}(\fs_{\ell})}-1} \Lis_{\fs_{\ell}}(\bu_{\ell}).\]
\end{theorem}

\begin{remark}\label{Rem: the set cT}
Precisely, we have
\begin{eqnarray*}
\Gamma_{\fs} \zeta_{A}(\fs)
&=& \sum_{\bj \in J_{\fs}} a_{\bj}(\theta) \Li_{\fs}(\bu_{\bj}) \\
&=& \sum_{\bj \in J_{\fs}} a_{\bj}(\theta) \sum_{P \in S^{r-1}} (-1)^{\nu(P)} \Lis_{P(\fs)}(P^{\times}(\bu_{\bj})) \\
&=& \sum_{\bj \in J_{\fs}} \sum_{P \in S^{r-1}} (-1)^{r-1} a_{\bj}(\theta) \cdot (-1)^{\dep(P(\fs))-1} \Lis_{P(\fs)}(P^{\times}(\bu_{\bj})),
\end{eqnarray*}
where we use the equality $\nu(P) + \dep(P(\fs)) = r$ for each $P \in S^{r-1}$. Let $T$ be the cardinality of the terms in the right hand side of the identity above. 
Then for convenience we renumber the indices $\ell$ of  $(b_{\ell}, \fs_{\ell}, \bu_{\ell})$ for which
\begin{equation}\label{E:triples}
\{ (b_{\ell}, \fs_{\ell}, \bu_{\ell}) | 1 \leq \ell \leq T \} = \{ ( (-1)^{r-1} a_{\bj}(\theta), P(\fs), P^{\times}(\bu_{\bj})) | \bj \in J_{\fs}, P \in S^{r-1} \},
\end{equation}
and $\dep(\fs_{\ell}) = 1$  for $1 \leq \ell \leq s$, and $\dep(\fs_{\ell}) \geq 2$  for $s+1 \leq \ell \leq T$. Note further that when $r=1$, ie., $\fs=s\in \NN$, we have $Li_{s}=\Lis_{s}$ and so the formula above for $\Gamma_{s}\zeta_{A}(s)$ is the same as Theorem~\ref{T:Chang}, which was established previously by Anderson-Thakur~\cite{AT90}.
\end{remark}

Note that the terms  $(-1)^{\tiny{\rm{dep}(\fs_{\ell})}-1} \Lis_{\fs_{\ell}}(\bu_{\ell})$ in the identity above occur as certain coordinates of the logarithm of the $t$-module considered in Theorem~\ref{theorem_log_li}.

\subsection{Proof of Theorem~\ref{Thm: Infty-adic LogMZV}}\label{Sec: proof of main thm}

Let $r$ be a positive integer and fix any index $\fs=(s_{1},\ldots,s_{r}) \in \NN^{r}$. Let $n:=\wt(\fs)$.  We identify the set of triples $(b_{\ell}, \fs_{\ell},\bu_{\ell})$ occurring  in Theorem~\ref{T:MZV as Lis} as the set
\[ \cT=\left\{1,\ldots,T \right\}, \]
where we understand that each element $\ell\in \cT$ corresponds to a triple $(b_{\ell}, \fs_{\ell},\bu_{\ell})$. We further rearrange the indices to decompose the disjoint union
\[\cT=\cT_{1}\cup \cT_{2} \]
so that $\cT_{1}$ consists of those indices $\ell$ for which $\dep (\fs_{\ell})=\dep(\bu_{\ell})=1$, and $\cT_{2}$ consists of those indices for which $\dep (\fs_{\ell})=\dep(\bu_{\ell}) >1$. 

Put $s:=|\cT_{1}|$ and note that due to cancellations of the right hand side of the identity in Theorem~\ref{T:MZV as Lis}, we allow $s$ to be either zero or $T$.

For each $\ell\in \cT$ equipped with $(b_{\ell}, \fs_{\ell},\bu_{\ell})$, we let $G_{\ell}$ be the $t$-module that is defined in (\ref{E:Explicit t-moduleCMPL}), and $\bv_{\ell}\in G_{\ell}(k)$ be the special point defined in (\ref{E:v_s,u}) that are constructed using the pair $(\widetilde{\fs_{\ell}},\widetilde{\bu_{\ell}})$, where $\widetilde{\cdot}$ is defined to reverse the order of components (see the definition in Theorem~\ref{theorem_log_li}). Note that $G_{\ell}$ is the $t$-module associated to the Anderson dual $t$-motive $\cM_{\ell}'$ that is associated to $(\widetilde{\fs_{\ell}},\widetilde{\bu_{\ell}} )$ and is defined in Remark~\ref{Rem: G and M}. So by Remark~\ref{Rem: R.A.T of M'}  $\cM_{\ell}'$ is rigid analytically trivial for each $\ell\in \cT$. Note that $\wt(\fs_{\ell})=\wt(\fs)=n$ for every $\ell\in \cT$. Therefore, by Theorem~\ref{theorem_log_li} the $n$th coordinate of $\log_{G_{\ell}}(\bv_{\ell})$ is
\[ (-1)^{\dep(\widetilde{\fs_{\ell}})-1}\Lis_{\widetilde{\widetilde{\fs_{\ell}}}} \left( \widetilde{\widetilde{\bu_{\ell}}}\right) = (-1)^{\dep(\fs_{\ell})-1}\Lis_{\fs_{\ell}} \left( \bu_{\ell}\right)  .\]

 Put $\cN:=C^{\otimes n}$, the $n$th tensor power of the Carlitz $t$-motive, and note that $\bC^{\otimes n}$ is its corresponding $t$-module (see~\cite[Sec.~5.2]{CPY14}). By the definition of $\bC^{\otimes n}$, we see that the $n$th coordinate of $\Lie \bC^{\otimes n }(\CC_{\infty})$ is tractable.

 Note that for $\ell\in \cT_{1}$, $\cM_{\ell}'$ is isomorphic to $\cN$ and for $\ell\in \cT_{2}$, $\cM_{\ell}'$ fits into the short exact sequence of left $\ok[t,\sigma]$-modules
\[0 \rightarrow \cN \rightarrow \cM_{\ell}' \rightarrow \cM_{\ell}'' \rightarrow 0, \]
where $\cM_{\ell}''$ is an Anderson dual $t$-motive. Let $\cM$ be the fiber coproduct of $\left\{ \cM_{\ell}' \right\}_{\ell=1}^{T}$ over $\cN$ and so by Proposition~\ref{Prop: R.A.T} $\cM$ is rigid analytically trivial. Let  $G_{\fs}$ be the $t$-module associated to $\cM$, ie., $G_{\fs}(\ok)\cong \cM/(\sigma-1)\cM$ as $\FF_{q}[t]$-modules. Hence $G_{\fs}$ is uniformizable by Remark~\ref{Rem: R.A.T}.

Recall that every $\bu_{\bj}$ belongs to $\mathbb{D}_{\fs}'$ for $\bj\in J_{\fs}$, and hence we have that for every $\ell\in \cT$, $\bu_{\ell}$ belongs to $\mathbb{D}_{\fs_{\ell}}'$ and hence $\widetilde{\bu_{\ell}}\in \mathbb{D}_{\widetilde{\fs_{\ell}}}'$. Since $(G_{\ell},\bv_{\ell})$ are constructed using $(\widetilde{\fs_{\ell}},\widetilde{\bu_{\ell}})$, which satisfy the conditions of  Theorem~\ref{theorem_log_li},   $\log_{G_{\ell}}(\bv_{\ell})$ converges $\infty$-adically for every $\ell\in \cT$.

Note that since all $\bu_{\ell}$ are integral points (see~(\ref{E:triples})), by Remark~\ref{Rem:Integral} the $t$-modules $\left\{G_{\ell}\right\}_{\ell=1}^{T}$ are defined over $k$ and hence $G_{\fs}$ is also defined over $k$ by Corollary~\ref{Cor: defined over A}.  

Now we let $\pi:\oplus_{\ell=1}^{T} G_{\ell} \rightarrow G_{\fs}$ be the morphism of $t$-modules over $k$ given in Definition~\ref{Def:pi}.  Recall that to simplify notation, we use $[a]$ for the action of $a\in \FF_{q}[t]$ on any $t$-module without confusion. For each $\ell\in \cT$, we  define 
\begin{equation}\label{E: Def of Z ell}
Z_{\ell}:=\log_{G_{\ell}}(\bv_{\ell})  \in \Lie G_{\ell} (\CC_{\infty}), 
\end{equation}
and  further set 
\begin{equation}\label{E: Def of Zs}
Z_{\fs}:=\partial \pi \left( (\partial[b_{\ell}(t)] Z_{\ell})_{\ell}\right)\in \Lie G_{\fs} (\CC_{\infty}),
\end{equation}
and 
\begin{equation}\label{E: Def of vs}
\bv_{\fs}:=\pi \left( ([b_{\ell}(t)] \bv_{\ell})_{\ell}\right)\in G_{\fs}(k),
\end{equation}
where $b_{\ell}\in A$ are given in Theorem~\ref{T:MZV as Lis}. We note that by the functional 
 equation~(\ref{E:IdentityExpLog})  we have
 \[\exp_{G_{\ell}}( Z_{\ell} ) = \bv_{\ell} \] 
and therefore by Lemma~\ref{Lemma: Key Lemma} we have 
\[\exp_{G_{\fs}}(Z_{\fs})=\bv_{\fs} .\]

On the other hand, by Theorem~\ref{theorem_log_li} the $n$th coordinate of $Z_{\ell}$ is given by $(-1)^{\tiny{\dep(\fs_{\ell}})-1}\Li_{\fs_{\ell}}^{\star}(\bu_{\ell}) $. By Lemma~\ref{Lemma: Key Lemma} and the formula in Theorem~\ref{T:MZV as Lis} we see that the $n$th coordinate of $Z_{\fs}$ is $\Gamma_{\fs}\zeta_{A}(\fs)$.

\begin{remark}
We mention that all other coordinates of $Z_{\fs}$ can be explicitly written down in terms of Taylor coefficients of $t$-motivic MZV's and $t$-motivic CMSPL's in~\cite{CGM19}. As the formulae of the coordinates are not used here, we refer the reader to~\cite{CGM19} in order to save some of length of this paper.
\end{remark}

\subsection{Examples}
\begin{example}

Take $q$ to be a power of any prime number $p$ and let $\fs = (1,1,2)$.
In this case, we have $\Gamma_{1} = \Gamma_{2} = 1$, $H_{1-1} = H_{2-1} = 1$, $J_{(1,1,2)} = \{ (0,0,0) \}$, $\bu_{(0,0,0)} = (1,1,1)$, $a_{(0,0,0)} = 1$.
Thus we have
\begin{eqnarray*}
\zeta_{A}(1,1,2) \!\!\! & = & \!\!\! \Li_{(1,1,2)}(1,1,1)
= \Lis_{(1,1,2)}(1,1,1) - \Lis_{(2,2)}(1,1) - \Lis_{(1,3)}(1,1) + \Lis_{4}(1) \\
\!\!\! & = & \!\!\! (-1)^{1-1} \Lis_{4}(1) + (-1)^{2-1} \Lis_{(1,3)}(1,1) \\
& & + (-1)^{2-1} \Lis_{(2,2)}(1,1) + (-1)^{3-1} \Lis_{(1,1,2)}(1,1,1),
\end{eqnarray*}
and $(b_{1}, \fs_{1}, \bu_{1}) = (1, 4, 1)$, $(b_{2}, \fs_{2}, \bu_{2}) = (1, (1,3), (1,1))$, $(b_{3}, \fs_{3}, \bu_{3}) = (1, (2,2), (1,1))$, $(b_{4}, \fs_{4}, \bu_{4}) = (1, (1,1,2), (1,1,1))$.

For $\ell = 1$, we have $G_{1} = \bC^{\otimes 4}$, and hence its $t$-action on $\GG_{a}^{4}$ is given by
\[
\bC^{\otimes 4}_{t} =  \left( \begin{array}{cccc} \theta & 1 & & \\ & \theta & 1 & \\ & & \theta & 1 \\ \tau & & & \theta \end{array} \right).
\]
We further have $\bv_{1} = (0, 0, 0, 1)^{\tr} \in \bC^{\otimes 4}(k)$, and $Z_{1} = (*, *, *, \Lis_{4}(1))^{\tr} \in \Lie \bC^{\otimes 4}(\CC_{\infty})$.
For $\ell = 2$, we have $G_{2} = \GG_{a}^{5}$ with the $t$-action
\[
[t] = \left( \begin{array}{cccc|c} \theta & 1 & & & \\ & \theta & 1 & & \\ & & \theta & 1 & \\ \tau & & & \theta & -\tau \\ \hline & & & & \theta + \tau \end{array} \right),
\]
and 
\[ \bv_{2} = (0, 0, 0, -1, 1)^{\tr} \in G_{2}(k), \]
\[ Z_{2} = (*, *, *, -\Lis_{(1,3)}(1,1), \Lis_{1}(1))^{\tr} \in \Lie G_{2}(\CC_{\infty}). \]
For $\ell = 3$, we have $G_{3} = \GG_{a}^{6}$ with the $t$-action
\[
[t] = \left( \begin{array}{cccc|cc} \theta & 1 & & & & \\ & \theta & 1 & & & \\ & & \theta & 1 & & \\ \tau & & & \theta & -\tau & \\ \hline & & & & \theta & 1 \\ & & & & \tau & \theta \end{array} \right),
\]
and points
\[ \bv_{3} = (0, 0, 0, -1, 0, 1)^{\tr} \in G_{3}(k), \]
\[ Z_{3} = (*, *, *, -\Lis_{(2,2)}(1,1), *, \Lis_{2}(1))^{\tr} \in \Lie G_{3}(\CC_{\infty}). \]
For $\ell = 4$, we have $G_{4} = \GG_{a}^{7}$ with the $t$-action
\[
[t] = \left( \begin{array}{cccc|cc|c} \theta & 1 & & & & & \\ & \theta & 1 & & & & \\ & & \theta & 1 & & & \\ \tau & & & \theta & -\tau & & \tau \\ \hline & & & & \theta & 1 & \\ & & & & \tau & \theta & - \tau \\ \hline & & & & & & \theta + \tau \end{array} \right),
\]
and 
\[ \bv_{4} = (0, 0, 0, 1, 0, -1, 1)^{\tr} \in G_{4}(k), \]
\[ Z_{4} = (*, *, *,  \Lis_{(1,1,2)}(1,1,1), *, -\Lis_{(1,1)}(1,1), *, \Lis_{1}(1))^{\tr} \in \Lie G_{4}(\CC_{\infty}). \]
Therefore we have $G_{(1,1,2)} = \GG_{a}^{10}$ with the $t$-action
\[
[t] = \left( \begin{array}{cccc|c|cc|ccc}
\theta & 1 & & & & & & & & \\
& \theta & 1 & & & & & & & \\
& & \theta & 1 & & & & & & \\
\tau & & & \theta & -\tau & -\tau & & -\tau & & \tau \\ \hline
& & & & \theta + \tau & & & & & \\ \hline
& & & & & \theta & 1 & & & \\
& & & & & \tau & \theta & & & \\ \hline
& & & & & & & \theta & 1 & \\
& & & & & & & \tau & \theta & - \tau \\
& & & & & & & & & \theta + \tau
\end{array} \right),
\]
and 
\[ \bv_{(1,1,2)} = \pi(\bv_{1}, \bv_{2}, \bv_{3}, \bv_{4}) = (0, 0, 0, 0, 1, 0, 1, 0, -1, 1)^{\tr} \in G_{(1,1,2)}(k), \]
\[ Z_{(1,1,2)} = (*, *, *, \zeta_{A}(1,1,2), \Lis_{1}(1), *, \Lis_{2}(1), *, -\Lis_{(1,1)}(1,1), \Lis_{1}(1))^{\tr} \in \Lie G_{(1,1,2)}(\CC_{\infty}). \]
\end{example}

\begin{example}
Take $q = 2$ and $\fs = (1,3)$. In this case, we have $\Gamma_{1} = 1$, $\Gamma_{3} = \theta^{2} + \theta$, $H_{1-1} = 1$, $H_{3-1} = t + \theta^{2}$, $J_{(1,3)} = \{ (0,0), (0,1) \}$, $\bu_{(0,0)} = (1,\theta^{2})$, $\bu_{(0,1)} = (1,1)$, $a_{(0,0)} = 1$, $a_{(0,1)} = t$.
Thus we have
\begin{eqnarray*}
(\theta^{2} + \theta) \zeta_{A}(1,3)
\!\!\! & = & \!\!\! \Li_{(1,3)}(1,\theta^{2}) + \theta \Li_{(1,3)}(1,1) \\
\!\!\! & = & \!\!\! \Lis_{(1,3)}(1,\theta^{2}) - \Lis_{4}(\theta^{2}) + \theta \Lis_{(1,3)}(1,1) - \theta \Lis_{4}(1) \\
\!\!\! & = & \!\!\! (-1)^{1-1} \Lis_{4}(\theta^{2}) + \theta \cdot (-1)^{1-1} \Lis_{4}(1) \\
& & + (-1)^{2-1} \Lis_{(1,3)}(1,\theta^{2}) + \theta \cdot (-1)^{2-1} \Lis_{(1,3)}(1,1),
\end{eqnarray*}
and $(b_{1}, \fs_{1}, \bu_{1}) = (1, 4, \theta^{2})$, $(b_{2}, \fs_{2}, \bu_{2}) = (\theta, 4, 1)$, $(b_{3}, \fs_{3}, \bu_{3}) = (1, (1,3), (1,\theta^{2}))$, $(b_{4}, \fs_{4}, \bu_{4}) =  (\theta, (1,3), (1,1)) $.

For $\ell = 1$, we have $G_{1} = \bC^{\otimes 4}$,
and points
\[ \bv_{1} = (0, 0, 0, \theta^{2})^{\tr} \in \bC^{\otimes 4}(k), \]
\[ Z_{1} = (*, *, *, \Lis_{4}(\theta^{2}))^{\tr} \in \Lie \bC^{\otimes 4}(\CC_{\infty}). \]
For $\ell = 2$, we have $G_{2} = \bC^{\otimes 4}$,
and points
\[ \bv_{2} = (0, 0, 0, 1)^{\tr} \in \bC^{\otimes 4}(k), \]
\[ Z_{2} = (*, *, *, \Lis_{4}(1))^{\tr} \in \Lie \bC^{\otimes 4}(\CC_{\infty}). \]
We also have
\[ [t] \bv_{2} = (0, 0, 1, \theta)^{\tr} \in \bC^{\otimes 4}(k). \]
For $\ell = 3$, we have $G_{3} = \GG_{a}^{5}$ with the $t$-action
\[
[t] = \left( \begin{array}{cccc|c} \theta & 1 & & & \\ & \theta & 1 & & \\ & & \theta & 1 & \\ \tau & & & \theta & - \theta^{2} \tau \\ \hline & & & & \theta + \tau \end{array} \right),
\]
and points
\[ \bv_{3} = (0, 0, 0, - \theta^{2}, 1)^{\tr} \in G_{3}(k), \]
\[ Z_{3} = (*, *, *, -\Lis_{(1,3)}(1,\theta^{2}), \Lis_{1}(1))^{\tr} \in \Lie G_{3}(\CC_{\infty}). \]
For $\ell = 4$, we have $G_{4} = \GG_{a}^{5}$ with the $t$-action
\[
[t] = \left( \begin{array}{cccc|c} \theta & 1 & & & \\ & \theta & 1 & & \\ & & \theta & 1 & \\ \tau & & & \theta & - \tau \\ \hline & & & & \theta + \tau \end{array} \right),
\]
and 
\[ \bv_{4} = (0, 0, 0, - 1, 1)^{\tr} \in G_{4}(k), \]
\[ Z_{4} = (*, *, *, -\Lis_{(1,3)}(1,1), \Lis_{1}(1))^{\tr} \in \Lie G_{4}(\CC_{\infty}). \]
We also have
\[ [t] \bv_{4} = (0, 0, 1, \theta+1, \theta+1)^{\tr} \in G_{4}(k). \]
Therefore we have $G_{(1,3)} = \GG_{a}^{6}$ with the $t$-action
\[
[t] = \left( \begin{array}{cccc|c|c}
\theta & 1 & & & & \\
& \theta & 1 & & & \\
& & \theta & 1 & & \\
\tau & & & \theta & - \theta^{2} \tau & -\tau \\ \hline
& & & & \theta + \tau & \\ \hline
& & & & & \theta + \tau
\end{array} \right),
\]
and 
\[ \bv_{(1,3)} = \pi(\bv_{1}, [t] \bv_{2} , \bv_{3}, [t] \bv_{4}) = (0, 0, 0, 1, 1, \theta+1)^{\tr} \in G_{(1,3)}(k), \]
\[ Z_{(1,3)} = (*, *, *, (\theta^{2} + \theta) \zeta_{A}(1,3), \Lis_{1}(1),  \theta \Lis_{1}(1)  )^{\tr} \in \Lie G_{(1,3)}(\CC_{\infty}). \]
\end{example}

\section{$v$-adic multiple zeta values}\label{Sec: v-adic MZV}
Throughout this section, we fix a finite place $v$ of $k$  corresponding to a monic irreducible polynomial of $A$ that is still denoted by $v$ for convenience,  and then fix an embedding $\ok\hookrightarrow \CC_{v}$.  Let $|\cdot|_{v}$ be the normalized $v$-adic absolute value on $\CC_{v}$. For a matrix $\gamma=(\gamma_{ij})$ with entries $\gamma_{ij}\in \CC_{v}$, we define
\[|\gamma|_{v}:=\max_{i,j}\left\{ |\gamma_{ij}|_{v}\right\} .\]

In this section, we will define $v$-adic multiple zeta values inspired by Furusho's definition of $p$-adic multiple zeta values in \cite{F04}. The primary goal of this section is to give a logarithmic interpretation for $v$-adic MZV's in Theorem~\ref{T: LogInt for v-adic MZV}. Together with Theorem~\ref{Thm: Introduction}, we apply Yu's sub-$t$-module theorem~\cite{Yu97} to prove Theorem~\ref{Thm: Well-defined map in Introduction} .

In what follows, for a $t$-module $G$ defined over $\ok$ we denote by $\log_{G}(\bx)_{v}$ the $v$-adic convergence value of $\log_{G}$ at $\bx\in G(\CC_{v})$ whenever $\log_{G}(\bx)_{v}$ converges, ie., $\log_{G}$ converges $v$-adically at $\bx$. 

\subsection{Definition of $v$-adic MZV's}
\subsubsection{The set up}\label{Subsec: Set up} Fix an index $\fs=(s_{1},\ldots,s_{r})\in  \NN^{r}  $ with $n:=\wt(\fs)$. As in Sec.~\ref{Sec: proof of main thm}, we identify the set of triples $(b_{\ell}, \fs_{\ell},\bu_{\ell})$ occurring  in Theorem~\ref{T:MZV as Lis} as the set
\[ \cT=\left\{1,\ldots,T \right\}, \]
where we understand that each element $\ell\in \cT$ corresponds to a triple $(b_{\ell},\fs_{\ell},\bu_{\ell})$. Recall that each $\bu_{\ell}$ is an integral point in $A^{\dep(\fs_{\ell})}$. We let $G_{\ell}$ be the $t$-module defined over $k$ (Sec.~\ref{Subsec: Gu}) and $\bv_{\ell}$ be the special point in $G_{\ell}(k)$ constructed using the pair $(\widetilde{\fs_{\ell}},\widetilde{\bu_{\ell}})$. We then let $G_{\fs}$ be the $t$-module associated to the fiber coproduct $\cM$ of the Anderson dual $t$-motives $\left\{ \cM_{\ell}'\right\}_{\ell=1}^{T}$ over $C^{\otimes n}$. Finally, we define $\bv_{\fs}:=\pi \left( ( [b_{\ell}(t)] \bv_{\ell} )_{\ell} \right)\in G_{\fs}(k)$. 

Note that $G_{\fs}$ has dimension $d:=d_{\fs}:=n+h_{s+1}+\cdots+h_{T}$ (see~(\ref{E:dimG})), where $n+h_{\ell}$ is the dimension of $G_{\ell}$ for $s+1\leq \ell \leq T$.

\subsubsection{$v$-adic analytic continuation of $\Lis_{\fs}$} For each $\ell \in \cT$, we consider the CMSPL $\Lis_{\fs_{\ell}}$ and  its $v$-adic convergence.  We note that $\Lis_{\fs_{\ell}}$ converges on the open unit ball centered at the zero of $\CC_{v}^{\dep(\fs_{\ell})}$ and it is shown in \cite[Sec.~4.1]{CM17} that $\Lis_{\fs_{\ell}}$ can be analytically continued to the closed unit ball centered at the zero of  $\CC_{v}^{\dep(\fs_{\ell})}$. Since $\bu_{\ell}$ is an integral point in $A^{\dep{\fs_{\ell}}}$, we have $|\bu_{\ell}|_{v}\leq 1$ and hence $\Lis_{\fs_{\ell}}$ is defined at $\bu_{\ell}$ in the sense of $v$-adic convergence. We denote by $\Lis_{\fs_{\ell}}(\bu_{\ell})_{v}$ the $v$-adic  convergence value of $\Lis_{\fs_{\ell}}$ at $\bu_{\ell}$, where we add the subscript $v$ to emphasize the $v$-adic convergence.  More precisely, $\Lis_{\fs_{\ell}}(\bu_{\ell})_{v}$ is the value $\frac{(-1)^{\dep(\fs_{\ell})-1}}{a(\theta)} $ multiplied by the  $n$th  coordinate of $\log_{G_{\ell}}([a] \bv_{\ell})_{v}$ for some nonzero polynomial $a\in \FF_{q}[t]$ with $| [a] \bv_{\ell} |_{v} < 1$. Since the coefficients of $\log_{G_{\ell}}$ are matrices with entries in $k$ (see \cite[(3.2.4)]{CM17}), we have $\Lis_{\fs_{\ell}}(\bu_{\ell})_{v} \in k_{v}$.

\subsubsection{The definition} Now we are ready to define $v$-adic MZV's using $\Lis_{\fs_{\ell}}(\bu_{\ell})_{v}$.

\begin{definition}\label{Def: v-adic MZV}  For any index $\fs=(s_{1},\ldots,s_{r}) \in \NN^{r}$, let notation be given in Theorem~\ref{T:MZV as Lis}. We define the $v$-adic MZV $\zeta_{A}(\fs)_{v}$ to be the following value:
\[  \zeta_{A}(\fs)_{v}:=\frac{1}{ \Gamma_{\fs}} \sum_{\ell}b_{\ell}\cdot (-1)^{\tiny{\rm{dep}(\fs_{\ell})}-1} \Lis_{\fs_{\ell}}(\bu_{\ell})_{v} \in k_{v}. \]
We call $wt(\fs):=\sum_{i=1}^{r} s_{i}$ the weight and $\dep (\fs):=r$ the depth of the presentation $\zeta_{A}(\fs)_{v}$. 
\end{definition}

We further mention that Thakur~\cite[Sec.~5.10]{T04} also defined $v$-adic MZV's by using Kummer congruences to interpolate the power sums at non-positive integers, and he remarked that his interpolated $v$-adic MZV's are not the same as ours defined above but they are expected to be related by certain linear relations.

\subsection{Logarithmic interpretation of $v$-adic MZV's} The primary goal in this subsection is to give a logarithmic interpretation for $\zeta_{A}(\fs)_{v}$, where the depth one case was established in \cite{AT90}.

\subsubsection{The $v$-adic convergence of $\log_{G_{\fs}}$} 

\begin{proposition}\label{Prop: v-adic conv of log G} Fix any index $\fs\in \NN^{r}$. For any  $\bx \in G_{\fs}(\CC_{v})$ with $|\bx|_{v} < 1$,  we have that $\log_{G_{\fs}}$ converges $v$-adically at $\bx$ in $\Lie G_{\fs}(\CC_{v})$. 
\end{proposition}

\begin{proof} We write $\log_{\bC^{\otimes n}}=\sum_{i=0}^{\infty}R_{i}\tau^{i}$ ($R_{i} \in \Mat_{n}(k)$) for the logarithm of $\bC^{\otimes n}$.
We denote the logarithms of $G_{\fs}$ and $G_{\ell}$ by
\[
\log_{G_{\fs}} = \sum_{i \geq 0} Q_{i} \tau^{i} \ (Q_{i} \in \Mat_{d}(k)) \ \ \ \mathrm{and} \ \ \ \log_{G_{\ell}} = \sum_{i \geq 0} Q_{\ell i} \tau^{i} \ (Q_{\ell i} \in \Mat_{n+h_{\ell}}(k))
\] respectively.  For each $s+1 \leq \ell \leq T$, we can write
\[
Q_{\ell i} = \left( \begin{array}{cc} R_{i} & R'_{\ell i} \\ & R''_{\ell i} \end{array} \right) \
(R'_{\ell i} \in \Mat_{n \times h_{\ell}}(k), \ R''_{\ell i} \in \Mat_{h_{\ell}}(k)).
\]
Then $Q_{i}$ is expressed as
\[
Q_{i} = \left( \begin{array}{ccccc} R_{i} & R'_{s+1, i} & R'_{s+2, i} & \cdots & R'_{T, i} \\ & R''_{s+1, i} & & & \\ & & R''_{s+2, i} & & \\ & & & \ddots & \\ & & & & R''_{T, i} \end{array} \right)
\]
for each $i$ since it forces the functional equation
\[\partial \pi \circ \left( \oplus_{\ell=1}^{T} \log_{G_{\ell}} \right)=\log_{G_{\fs}} \circ \pi   .\]

Let $\bx = (\bx_{0}^{\tr}, \bx_{s+1}^{\tr}, \bx_{s+2}^{\tr}, \dots, \bx_{T}^{\tr})^{\tr} \in G_{\fs}(\CC_{v})$ with $\bx_{0} \in \CC_{v}^{n}$, $\bx_{\ell} \in \CC_{v}^{h_{\ell}}$ for $s+1 \leq \ell \leq T$, and $|\bx|_{v} < 1$.
Note that  by \cite[Sec.~3.3]{CM17}  we have
\[ \left\{   \left| Q_{s+1, i} \left( \begin{array}{c} \mathbf{x}_{0} \\ \bx_{s+1} \end{array} \right)^{(i)} \right|_{v} \right\} \to 0 \hbox{ and } \left\{   \left| Q_{\ell i} \left( \begin{array}{c} \mathbf{0} \\ \bx_{\ell} \end{array} \right)^{(i)} \right|_{v} \right\} \to 0 \hbox{ as } i \to \infty.
\]
It follows that
\[
|Q_{i} \bx^{(i)}|_{v} \leq \max_{s+2 \leq \ell \leq T} \left\{   \left| Q_{s+1, i} \left( \begin{array}{c} \mathbf{x}_{0} \\ \bx_{s+1} \end{array} \right)^{(i)} \right|_{v},  \left| Q_{\ell i} \left( \begin{array}{c} \mathbf{0} \\ \bx_{\ell} \end{array} \right)^{(i)} \right|_{v} \right\} \to 0 \hbox{ as } i \to \infty.
\]
\end{proof}

\begin{proposition}\label{Prop: conv at av}
For any index $\fs\in \NN^{r}$, we continue with the notation as above. Then there is a precise nonzero polynomial $a\in \FF_{q}[t]$ (depending on $\fs$ and $v$) so that $|[a]\bv_{\fs}|_{v} < 1$, hence $\log_{G_{\fs}}\left( [a]\bv_{\fs}\right)_{v}$ converges. 
\end{proposition}

\begin{proof}
Write $\fs_{\ell} = (s_{\ell 1}, \dots, s_{\ell r_{\ell}})$ and set
\[
a_{\ell} := (v(t)^{s_{\ell 1} + s_{\ell 2} + \cdots + s_{\ell r_{\ell}}}-1) (v(t)^{s_{\ell 1} + s_{\ell 2} + \cdots + s_{\ell, r_{\ell}-1}}-1) \cdots (v(t)^{s_{\ell 1}}-1) \in \FF_{q}[t]
\]
and
\[
a := \prod_{\ell=1}^{T} a_{\ell} \in \FF_{q}[t],
\]
where $v(t) := v|_{\theta = t}$. Since $\bu_{\ell}\in A^{\dep(\fs_{\ell})}$ for each $\ell$, by Remark~\ref{Rem:Integral} we have that for each $\alpha\in \FF_{q}[t]$, the coefficient matrices  of $\tau^{i}$ of $[\alpha]$ are in $\Mat_{\dim G_{\ell}}(A)$. It follows that
$| [a]([b_{\ell}(t)] \bv_{\ell}) |_{v}=| [b_{\ell}(t)]([a] \bv_{\ell}) |_{v} \leq | [a] \bv_{\ell} |_{v} \leq | [a_{\ell}] \bv_{\ell} |_{v} < 1$, where the last inequality comes from the proof of \cite[Prop.\ 4.1.1]{CM17}. So by \cite[Sec.~3.3]{CM17}  again $\log_{G_{\ell}} ([a] ([b_{\ell}] \bv_{\ell}))_{v}$ converges in $\Lie G_{\ell}(\CC_{v})$.
Therefore we have
\begin{equation}\label{E: v-adic inequality}
|[a] \bv_{\fs}|_{v}=|\pi\left(([a][b_{\ell}(t)] \bv_{\ell})_{\ell}  \right)|_{v}  \leq \max_{\ell} \left\{ | [a]([b_{\ell}(t)] \bv_{\ell}) |_{v} \right\} < 1,
\end{equation}
where the first inequality comes from Definition~\ref{Def:pi}. It follows that
\[
\log_{G_{\fs}}([a] \bv_{\fs})_{v} = \partial \pi ((\log_{G_{\ell}}([a]([b_{\ell}(t)] \bv_{\ell}))_{v})_{\ell})
\]
converges in $\Lie G_{\fs}(\CC_{v})$.
\end{proof}

\begin{theorem}\label{T: LogInt for v-adic MZV} Fix a finite place $v$ of $k$. Given an index $\fs = (s_{1}, \dots, s_{r}) \in \NN^{r}$, we put $n := \wt(\fs)$ and let $\left\{(b_{\ell}, \fs_{\ell}, \bu_{\ell})\right\}_{\ell=1}^{T}$ be the set of triples in (\ref{E:triples}).
Let $G_{\ell}$ be the $t$-module defined over $k$ and $\bv_{\ell} \in G_{\ell}(k)$ be the special point which are constructed using the pairs $(\widetilde{\fs_{\ell}},\widetilde{\bu_{\ell}})$, and $G_{\fs}$ be the $t$-module over $k$ and $\bv_{\fs}\in G_{\fs}(k)$ be constructed as above. We  take a nonzero $a \in \FF_{q}[t]$ for which $|[a] \bv_{\fs}|_{v} < 1$.  Then the $n$th coordinate of $\log_{G_{\fs}}([a] \bv_{\fs})_{v}$ is given  by $a(\theta)\Gamma_{\fs}\zeta_{A}(\fs)_{v}$.
\end{theorem}

\begin{remark}
Since the $n$th coordinate of $\Lie G_{\fs}(\CC_{v})$ is tractable, it is enough to show that the statement of Theorem 6.2.4 holds for some $a$. Indeed, assume that the statement holds for $a$, and let $a' \in \Fq[t]$ be another nonzero polynomial with $|[a'] \bv_{\fs}|_{v} < 1$. Then we have
\begin{eqnarray*}
a(\theta) \times n\mathrm{th \ coordinate \ of \ } \log_{G_{\fs}}([a'] \bv_{\fs})_{v}
&=& n\mathrm{th \ coordinate \ of \ } \log_{G_{\fs}}([a] [a'] \bv_{\fs})_{v} \\
&=& a'(\theta) \times n\mathrm{th \ coordinate \ of \ } \log_{G_{\fs}}([a] \bv_{\fs})_{v} \\
&=& a(\theta) a'(\theta) \Gamma_{\fs} \zeta_{A}(\fs)_{v}.
\end{eqnarray*} 
\end{remark}

{\textit{ Proof of Theorem~\ref{T: LogInt for v-adic MZV}}}. Let $\cT=\left\{ 1,\cdots,T \right\}$ be given as before in Sec.~\ref{Subsec: Set up}.  We first take a nonzero polynomial $a\in \FF_{q}[t]$ so that 
\begin{enumerate}
\item[$\bullet$] $|[a(t)]\bv_{\fs}|_{v}<1$.
\item[$\bullet$] $|[a(t)] \bv_{\ell}|_{v}<1$ for all $\ell\in \cT$.
\end{enumerate}
Note that the second property can be obtained using the same arguments in (\ref{E: v-adic inequality}). It follows by Proposition~\ref{Prop: v-adic conv of log G} that  $\log_{G_{\fs}}([a] \bv_{\fs})_{v}$ converges, and by \cite[Thm.~3.3.3]{CM17} that every $\log_{G_{\ell}}([a]([b_{\ell}(t)] \bv_{\ell}))_{v}$ converges for every $\ell \in \cT$.  We have seen that
\begin{enumerate}
\item[$\bullet$] $a(\theta)b_{\ell} \times (-1)^{\dep(\fs_{\ell})-1}\Li_{\fs_{\ell}}^{\star} (\bu_{\ell})_{v}$ is the $n$th coordinate of $\log_{G_{\ell}}([a]([b_{\ell}(t)] \bv_{\ell}))_{v}$ (see~\cite[Def.~4.1.2]{CM17});
\item[$\bullet$] The $n$th coordinate of $\Lie G_{\ell}$ is tractable (see~(\ref{E:Explicit t-moduleCMPL})).
\end{enumerate}

Recall by~(\ref{E:logFunctorial}) that we have the following functional equation:
\begin{equation}\label{E:logFEQ for Gs}
\log_{G_{\fs}}\circ \pi=\partial \pi \circ \left( \oplus_{\ell=1}^{T}  \log_{G_{\ell}} \right).
\end{equation}

Recall by (\ref{E: Def of vs}) $\bv_{\fs}:=\pi \left( ([b_{\ell}(t)] \bv_{\ell})_{\ell}\right)\in G_{\fs}(k)$. Now we consider the specialization at the point $([a]([b_{\ell}(t)] \bv_{\ell}))_{\ell} \in\oplus_{\ell=1}^{T} G_{\ell}(\CC_{v}) $ of both sides of (\ref{E:logFEQ for Gs}) under the $v$-adic convergence. The LHS of (\ref{E:logFEQ for Gs}) evaluated at $([a]([b_{\ell}(t)] \bv_{\ell}))_{\ell}$ is the vector $\log_{G_{\fs}}([a] \bv_{\fs})_{v}$, which is identical to $\partial \pi \left( 
(\log_{G_{\ell}}\left([a]([b_{\ell}(t)] \bv_{\ell}))\right)_{\ell}\right)$ from the RHS of  (\ref{E:logFEQ for Gs}) evaluated at $([a]([b_{\ell}(t)] \bv_{\ell}))_{\ell}$.  By Definition~\ref{Def:pi} we see that the $n$th coordinate of $\partial \pi \left( 
(\log_{G_{\ell}}([a]([b_{\ell}(t)] \bv_{\ell})))_{\ell}\right)$ is given by
\[\sum_{\ell=1}^{T} n\hbox{th coordinate of }\log_{G_{\ell}}\left([a]([b_{\ell}(t)] \bv_{\ell})\right),\]
which is exactly the value (by \cite[Def.~4.1.2]{CM17} and Definition~\ref{Def: v-adic MZV})
\[ \sum_{\ell=1}^{T} a(\theta)\times b_{\ell}(\theta)\cdot (-1)^{\dep(\fs_{\ell})-1}   \Lis_{\fs_{\ell}}(\bu_{\ell})_{v} =a(\theta)\Gamma_{\fs}\zeta_{A}(\fs)_{v} .\]

\subsection{Review of Yu's sub-$t$-module theorem}\label{Subsec: Yu's sub-t-module thm}
 The following notion of {\it{regular}} $t$-modules is due to Yu~\cite[p.~218]{Yu97}.
\begin{definition} Let $G$ be a $t$-module defined over $\ok$. We say that $G$ is regular if there is a positive integer $\nu$ for which the $a$-torsion submodule of $G(\ok)$ is free of rank $\nu$ over $\FF_{q}[t]/(a)$ for every nonzero polynomial $a\in \FF_{q}[t]$. 
\end{definition}

Note that every $n$th tensor power of Carlitz module $\bC^{\otimes n}$ and Drinfeld modules defined over $\ok$ are regular~\cite[p.~217]{Yu97}. Other examples of regular $t$-modules arising from special $\Gamma$-values, see~\cite{S97, BP02}.

\begin{proposition} Given an index $\fs\in \NN^{r}$, we let $G_{\fs}$ be the $t$-module constructed in Section~\ref{Subsec: Set up} .
Then $G_{\fs}$ is regular. 
\end{proposition}

\begin{proof} Note that $G_{\fs}$ is the $t$-module associated to the rigid analytically trivial Anderson dual $t$-motive $\cM$. Let $\nu$ be the rank of $\cM$ over $\ok[t]$. Since $G_{\fs}$ is uniformizable by Proposition~\ref{Prop: R.A.T} and Remark~\ref{Rem: R.A.T}, we have the following $\FF_{q}[t]$-module isomorphism via $\exp_{G_{\fs}}$:
\[ \Lie G_{\fs}(\CC_{\infty})\big/ \Lambda_{\fs} \cong G_{\fs}(\CC_{\infty}), \]
where $\Lambda_{\fs}:=\Ker \exp_{G_{\fs}} \subset \Lie G_{\fs}(\CC_{\infty})\cong \CC_{\infty}^{\dim G_{\fs}}$ is a discrete free $\FF_{q}[t]$-submodule of rank $\nu$ by \cite[Thm.~5.28]{HJ16} (cf.~\cite[Thm.~4]{A86}). It follows that the $a$-torsion submodule of $G_{\fs}(\ok)$ is isomorphic to 
\[\partial [a(t)]^{-1}\Lambda_{\fs}\big/\Lambda_{\fs}\cong \left(\FF_{q}[t]/(a) \right)^{\nu} \] for any nonzero polynomial $a\in \FF_{q}[t]$. 
\end{proof}
Let $G$ be a $t$-module defined over $\ok$. Note that $G$ is also regarded as a linear algebraic group over $\bar{k}$. Any connected linear algebraic subgroup of $G$ that is defined over $\ok$ and that is invariant under the $\FF_{q}[t]$-action is called a sub-$t$-module of $G$ over $\ok$.  The spirit of Yu's sub-$t$-module theorem stated in Theorem~\ref{Thm: Yu Introduction} is that for a given logarithmic vector $Z$ of an algebraic point on a regular $t$-module $G$ defined over $\ok$, the smallest $\partial[t]$-invariant vector subspace over $\ok$ in $\Lie G(\CC_{\infty})$ containing that logarithmic vector must
be $\Lie H(\CC_{\infty})$ for some sub-$t$-module $H\subset G$ over $\ok$.

 \subsection{The main result} 
We call a positive integer $n$ \lq\lq even\rq\rq \ if $n$ is divisible by $q-1$; otherwise we call $n$  \lq\lq odd\rq\rq.  As mentioned in the introduction, $(1-v^{-n})\zeta_{A}(n)_{v}$ is identical to Goss' $v$-adic zeta value at $n$, by~\cite{Go79} we know that $\zeta_{A}(n)_{v}=0$ for $n$ \lq\lq even\rq\rq, and by~\cite{Yu91} $\zeta_{A}(n)_{v}$ is transcendental over $k$ for $n$ ``odd''. The main result of this section is as follows.

\begin{theorem}\label{Thm: Well-defined map} Let $v$ be a finite place of $k$ and
 fix a positive integer  $n$. Let $\overline{\cZ}_{n}$ be the $\ok$-vector space spanned by all $\infty$-adic MZV's of weight $n$, and $\overline{\cZ}_{n,v}$ be the $\ok$-vector space spanned by all $v$-adic MZV's  of weight $n$. Then we have a well-defined surjective $\ok$-linear map
\[\overline{\cZ}_{n} \twoheadrightarrow \overline{\cZ}_{n,v} \]
given by 
\[\zeta_{A}(\fs)\mapsto \zeta_{A}(\fs)_{v}, \]
and its kernel contains the one-dimensional vector space $\ok\cdot \zeta_{A}(n)$ when $n$ is ``even''.
\end{theorem}

In other words, the theorem above shows that the $v$-adic MZV's of weight $n$ satisfy the same $\ok$-linear relations that their corresponding $\infty$-adic MZV's of weight $n$ satisfy. 

\begin{proof}[Proof of Theorem~\ref{Thm: Well-defined map}]
Suppose that we have a non-trivial $\ok$-linear relation among some MZV's of weight $n$
\[ c_{1}\zeta_{A}(\fs_{1})+\cdots +c_{m}\zeta_{A}(\fs_{m}) =0,\]
which we rewrite as 
\begin{equation}\label{E:LinearRe}
\epsilon_{1}\Gamma_{\fs_{1}}\zeta_{A}(\fs_{1})+\cdots +\epsilon_{m}\Gamma_{\fs_{m}}\zeta_{A}(\fs_{m})=0 ,
\end{equation}
where $\left\{ \epsilon_{i}:=\frac{c_{i}}{\Gamma_{\fs_{i}}} \right\}_{i=1}^{m}$ are not all zero. For each index $\fs_{i}$, let $G_{\fs_{i}}$ be the $t$-module defined over $k$, $\bv_{\fs_{i}}\in G_{\fs_{i}}(k)$ be the special point and $Z_{\fs_{i}}\in \Lie G_{\fs_{i}}(\CC_{\infty})$ be the vector given in Theorem~\ref{Thm: Introduction}. We identify $\Lie G_{\fs_{i}}$ with ${\mathbb{A}^{\dim G_{\fs_{i}}}}_{/ k}$, the affine variety of dimension $g_{i}:=\dim G_{\fs_{i}}$ over $k$, and let $\bX_{i}:=\left(X_{i1},  \ldots,X_{i g_{i}}\right)^{\tr}$ be the coordinates of $\Lie G_{\fs_{i}}$.   Let $G:=G_{\fs_{1}}\oplus \cdots \oplus G_{\fs_{m}}$ be the $t$-module as direct sum of $\left\{  G_{\fs_{i}}\right\}_{i=1}^{m}$ and so $\Lie G$ is identified with ${\mathbb{A}^{g_{1}+\cdots+g_{m}}}_{/ k}$ with coordinates 
\[\bX=\left( \bX_{1}^{\tr},\ldots,\bX_{m}^{\tr} \right)^{\tr}. \]

Let $V$ be the smallest linear subspace of $\Lie G(\CC_{\infty})$ defined over $\ok$ for which
\begin{enumerate}
\item[$\bullet$] $V$ contains the vector $Z:=\left(Z_{\fs_{1}}^{\tr},\ldots,Z_{\fs_{m}}^{\tr}  \right)^{\tr} \in \Lie G (\CC_{\infty})$.
\item[$\bullet$] $V$ is invariant under the $\partial [t]$-action.
\end{enumerate} 
We define  the following hyperplane over $\ok$
\[ W:=\left\{ \epsilon_{1}X_{1n}+\cdots+\epsilon_{m}X_{mn}=0 \right\} \subset \Lie G,\]
where we simply use $\Lie G$ for the base change of $\Lie G$ over $\ok$ when it is clear from the contents, and note that $Z\in W(\CC_{\infty})$. We further note that since the $n$th coordinate of $\Lie G_{\fs_{i}}$ is tractable for each $i$, $W$ is invariant under $\partial[t]$ and we see that $V\subset W(\CC_{\infty})$. By Theorem~\ref{Thm: Yu Introduction} there exists a sub-$t$-module $H\subset G$ over $\ok$ for which 
\[V= \Lie H(\CC_{\infty}). \]
Let $\mathcal{V}\subset \Lie G$ be the linear sub-variety underlying $V$. That is, $\mathcal{V}$ is the variety defined by the defining equations of $V\subset \Lie G (\CC_{\infty})$ over $\ok$. So we have $\mathcal{V}(\CC_{\infty})=V=\Lie H(\CC_{\infty})$, and hence $\mathcal{V}=\Lie H$.

For each $\fs_{i}$, by Proposition~\ref{Prop: conv at av} we are able to pick a nonzero $a_{i}\in \FF_{q}[t]$ for which $|[a_{i}]\bv_{\fs_{i}}|_{v}<1$. Put $a:=\prod_{i=1}^{m}a_{i}$. Note that by Corollary~\ref{Cor: defined over A} the action $[t]$ on each $G_{\fs_{i}}$ has coefficient matrices with entries in $A$. Therefore, we have that 
\[ |[a]\bv_{\fs_{i}}|_{v}\leq |[a_{i}]\bv_{\fs_{i}}|_{v}<1, \]
and hence by Proposition~\ref{Prop: v-adic conv of log G}   $\log_{G_{\fs_{i}}}([a] \bv_{\fs_{i}})_{v}$ converges for each $1\leq i \leq m$. Define
\[\bv:=
 \begin{pmatrix}
\bv_{\fs_{1}}\\
\vdots\\
\bv_{\fs_{m}} \\
\end{pmatrix} \in G(k).
\]

We claim that for each $\epsilon > 0$, there exists $\ell > 0$ such that $|[v(t)^{\ell}] ([a] \bv)|_{v} < \epsilon$.
By the construction, $G$ is an iterated extension of tensor powers of the Carlitz $t$-module.
More precisely, let $\be_{1}, \dots, \be_{d}$ be the standard basis of $G = \GG_{a}^{d}$ (for some $d\in \NN$).
Then there exists a sequence $I_{0} := \emptyset \subsetneq I_{1} \subsetneq \cdots \subsetneq I_{f} := \{ 1, \dots, d \}$ such that
if we set $G_{f}:=G$ and $G_{i} := \GG_{a}^{|I_{i}|}$ with basis $\be_{j}$ ($j \in I_{i}$), then $G_{i}$ forms a quotient $t$-module of $G$ defined over $A$, and the kernel of the natural projection $G_{i} \to G_{i-1}$ is a tensor power of $\bC$ for each $1 \leq i \leq f$:
\[
0 \to \bC^{\otimes (|I_{i}| - |I_{i-1}|)} \to G_{i} \to G_{i-1} \to 0.
\]

We consider the image of $[a] \bv$ in $G_{i}(A)$ via the natural projection $G \to G_{i}$ and then modulo powers of $v$. Using an inductive argument on $i$, it suffices to show that for each $s \in \NN$ and each $\bx \in \bC^{\otimes s}(\ok)$ with $|\bx|_{v} < 1$, there exists $\ell > 0$ such that $|[v(t)^{\ell}]_{s} \bx|_{v} < \epsilon$.
By \cite[Proposition 1.6.1]{AT90}, we have $[v(t)^{s}]_{s} = \tau^{\deg v} + v \alpha$ for some $\alpha \in \Mat_{s}(A[\tau])$.
Therefore, for large $\ell > 0$ divisible by $s$, we have $|[v(t)^{\ell}]_{s} \bx|_{v} < \epsilon$.

Since $V$ is invariant under the $\partial[t]$-action, we have that  $\partial [a] Z \in \mathcal{V}(\CC_{\infty})=\Lie H(\CC_{\infty})$,  whence 
\[ \exp_{G}\left( \partial[a] Z \right)= [a] \bv \in  G(k)\cap H(\ok)\subset H(\ok) .\]
Let $\phi: H\hookrightarrow G$ be the natural embedding morphism of $t$-modules. Note that $H$ is a linear algebraic group and so it is smooth. Since $H$ is smooth, $p$-torsion, commutative, affine and connected, by \cite[Lemma B.1.10]{CGP10}, $H$ is isomorphic to $\GG_{a}^{h}$ for some $h$ over $\ok$.
Fix an isomorphism $H \cong \GG_{a}^{h}$ over $\ok$. By using this identification, we write
\[
\phi = \sum_{j=0}^{N} A_{j} \tau^{j} \colon \Ga^{h} \cong H \hookrightarrow G = \Ga^{g_{1} + \cdots + g_{m}}, \ \ \ A_{j} \in \Mat_{(g_{1} + \cdots + g_{m}) \times h}(\ok)
\]
and
\[
\log_{H} = \sum_{i \geq 0} B_{i} \tau^{i}, \ \ \ B_{i} \in \Mat_{h}(\ok).
\]
We also write
\[
\log_{G} = \sum_{i \geq 0} C_{i} \tau^{i}, \ \ \ C_{i} \in \Mat_{g_{1} + \cdots + g_{m}}(\ok).
\]
By the functional equation (\ref{E:logFunctorial}), we have
\[
A_{0} \circ \left( \sum_{i \geq 0} B_{i} \tau^{i} \right) = \left( \sum_{i \geq 0} C_{i} \tau^{i} \right) \circ \left( \sum_{j=0}^{N} A_{j} \tau^{j} \right)=\sum_{\ell \geq 0}\left( \sum_{0\leq j\leq N, i+j=\ell} C_{i}A_{j}^{(i)}\right)\tau^{\ell}
\]
as formal power series. 

We take $0<\epsilon<1$ for which the power series $\sum_{i \geq 0} C_{i} \tau^{i}$ converges on the domain 
\[\left\{ \bx\in \Lie G(\CC_{v})| \hbox{ } |\bx|_{v}<\epsilon \right\}  .\]
According to the claim above, we are able to take  an $\ell > 0$  and then replace $a$ by $v(t)^{\ell} a$ if necessary so that

\[\max_{0\leq j \leq N}\left\{\bigl|A_{j}([a]\bv)^{(j)} \bigr|_{v} \right\}< \epsilon,\]
whence \[ \Bigl|C_{i}\left(A_{j}([a]\bv)^{(j)} \right)^{(i)} \Bigr|_{v}\rightarrow 0 \hbox{ as }i\rightarrow \infty \]
for every $0\leq j\leq N$. Thus, we have the following identity 
\[ \left( \left( \sum_{i \geq 0} C_{i} \tau^{i} \right) \circ \left( \sum_{j=0}^{N} A_{j} \tau^{j} \right)\right) ([a]\bv)=  \left( \sum_{i \geq 0} C_{i} \tau^{i} \right) \left( \left( \sum_{j=0}^{N} A_{j} \tau^{j} \right) ([a]\bv) \right).\]

Since $[a] \bv$ is $v$-adically small enough from the above, and $H$ is invariant under the $[a]$-action, we can pull back $[a] \bv$ via the embedding $\phi$, and the above functional equation among formal power series implies the corresponding equality among vectors over $\CC_{v}$:
\begin{align*}
\left( \log_{G_{\fs_{i}}}([a] \bv_{\fs_{i}})_{v} \right)_{i} &= \log_{G}([a] \bv)_{v} = \sum_{i \geq 0} C_{i} \tau^{i} \left( \sum_{j=0}^{N} A_{j} \tau^{j}([a] \bv) \right) \\
&= \left( \left( \sum_{i \geq 0} C_{i} \tau^{i} \right) \circ \left( \sum_{j=0}^{N} A_{j} \tau^{j} \right) \right) ([a] \bv)
= \left( A_{0} \circ \left( \sum_{i \geq 0} B_{i} \tau^{i} \right) \right) ([a] \bv) \\
&= A_{0} \sum_{i \geq 0} B_{i} \tau^{i} ([a] \bv)
= \log_{H}([a] \bv)_{v} \in \Lie H(\CC_{v}).
\end{align*}

Note that since $\mathcal{V}(\CC_{\infty})=V\subset W(\CC_{\infty})$, we have  $\Lie H =\mathcal{V} \subset W$. It follows that  the vector $\left( \log_{G_{\fs_{i}}}([a] \bv_{\fs_{i}})_{v} \right)_{i}$ belongs to $W(\CC_{v})$, whence  satisfying the $\ok$-linear relations
\[ \epsilon_{1}X_{1n}+\cdots+\epsilon_{m}X_{mn}=0    .\]
By Theorem~\ref{T: LogInt for v-adic MZV} the $n$th coordinate of $\log_{G_{\fs_{i}}}([a] \bv_{\fs_{i}})$ is the value $a(\theta)\Gamma_{\fs_{i}}\zeta_{A}(\fs_{i})_{v}$ and hence we obtain the desired identity
 \[ c_{1}\zeta_{A}(\fs_{1})_{v}+\cdots+c_{m}\zeta_{A}(\fs_{m})_{v}=0 .\]

Note that by \cite{Go79} we have $\zeta_{A}(n)_{v}=0$ for $n\in \NN$ \lq\lq even\rq\rq \ as our $v$-adic zeta value at $n$ is  Goss $v$-adic zeta value at $n$ multiplied by $(1-v^{-n})^{-1}$ (see \cite[Theorem 3.8.3. (II)]{AT90}). Therefore the second assertion follows immediately. 
\end{proof}

\begin{corollary}\label{Cor: Consequence} Let $v$ be a finite place of $k$. Set $\overline{\cZ}_{0} := \overline{\cZ}_{0,v} := \ok$. Let $\overline{\cZ}:=\sum_{n=0}^{\infty} \overline{\cZ}_{n}$ be the $\ok$-vector space spanned by all MZV's, and let $\overline{\cZ}:=\sum_{n=0}^{\infty} \overline{\cZ}_{n,v}$ be the  $\ok$-vector space spanned by all $v$-adic MZV's. Then we have the $\ok$-linear map
\[\overline{\cZ} \twoheadrightarrow \overline{\cZ}_{v} \]
given by 
\[\zeta_{A}(\fs)\mapsto \zeta_{A}(\fs)_{v}. \]
\end{corollary}

\begin{proof}
By \cite[Thm.~2.2.1]{C14}, we have a natural isomorphism $\oplus_{n} \overline{\cZ}_{n} \cong \overline{\cZ}$ of $\ok$-algebras.
Thus the $\ok$-linear maps $\overline{\cZ}_{n} \twoheadrightarrow \overline{\cZ}_{n,v}$ in Theorem \ref{Thm: Well-defined map} imply the $\ok$-linear map \[ \overline{\cZ} \cong \oplus_{n} \overline{\cZ}_{n} \twoheadrightarrow \oplus_{n} \overline{\cZ}_{n,v} \twoheadrightarrow \overline{\cZ}_{v}. \]
\end{proof}

\begin{remark}\label{Rem: future project}
Based on the results above, the following are some natural questions which need additional work.
\begin{enumerate}
\item Does $\overline{\cZ}_{v}$ have an algebra structure?
\item For each positive integer $n$, what is the kernel of the above map $\overline{\cZ}_{n}\twoheadrightarrow \overline{\cZ}_{n,v}$?
\item Is the above map $\overline{\cZ}\twoheadrightarrow \overline{\cZ}_{v}$ an algebra homomorphism? If so, what is its kernel?
\end{enumerate}
\end{remark}

\bibliographystyle{alpha}

\begin{thebibliography}{999999}


\bibitem[A86]{A86}
G.\  W.\  Anderson, \textit{$t$-motives}, Duke Math. J. \textbf{53} (1986), no. 2, 457--502.

\bibitem[A94]{A94}
G.\  W.\  Anderson, \textit{Rank one elliptic $A$-modules and $A$-harmonic series}, 
Duke Math. J. \textbf{73} (1994), no. 3, 491-542.

\bibitem[ABP04]{ABP04}
G. W. Anderson, W. D. Brownawell and M. A. Papanikolas, \textit{Determination of the algebraic relations among special $\Gamma$-values in positive characteristic}, Ann. of Math. (2) \textbf{160} (2004), no. 1, 237--313.


\bibitem[AT90]{AT90}
G.\ W.\ Anderson and D.\ S.\ Thakur, \textit{Tensor powers of the Carlitz module and zeta values}, Ann. of Math. (2) \textbf{132} (1990), no. 1, 159--191.

 \bibitem[AT09]{AT09}
G.\ W.\ Anderson and D.\ S.\ Thakur, \textit{Multizeta values for $\FF_{q}[t]$, their period interpretation, and relations between them},  Int. Math. Res. Not. IMRN (2009), no. 11, 2038--2055.

\bibitem[An04]{An04} Y.\ Andr\'e, \textit{ Une introduction aux motifs (motifs purs, motifs mixtes, p\'eriodes)}, Panoramas et Synth\'eses, \textbf{17}. Soci\'et\'e Math\'ematique de France, Paris, 2004.

\bibitem[BGF19]{BGF19} J.\ I.\ Burgos Gil and J.\ Fres\'an, \textit{Multiple zeta values: from numbers to motives}, to appear in 
Clay Mathematics Proceedings.

\bibitem[BP02]{BP02} W.\ D.\ Brownawell and M.\ A.\ Papanikolas, \textit{Linear independence of gamma values in positive characteristic},  J. Reine Angew. Math. 549 (2002), 91-148. 

\bibitem[BP16]{BP16} W.\ D.\ Brownawell and M.\ A.\ Papanikolas, \textit{A rapid introduction to Drinfeld modules, t-modules, and t-motives}, available at http://www.math.tamu.edu/~map/BanffSurveyRev2.pdf. 


\bibitem[BW07]{BW07}
A.\ Baker and G.\ W\"ustholz,  \textit{Logarithmic forms and Diophantine geometry}, New Mathematical Monographs, 9. Cambridge University Press, Cambridge, 2007. 

\bibitem [Ca35]{Ca35}
L.\ Carlitz, \textit{On certain functions connected with polynomials in a Galois field}, Duke Math. J. \textbf{1}
  (1935), no. 2, 137-168.


\bibitem[C14]{C14}
C.\-Y.\ Chang, \textit{Linear independence of monomials of multizeta values in positive characteristic}, Compositio Math. \textbf{150} (2014), 1789-1808.


\bibitem[C16]{C16}
C.\-Y.\ Chang, \textit{Linear relations among double zeta values in positive characteristic}, 
Camb. J. Math. \textbf{4} (2016), no. 3, 289-331. 

\bibitem[CM19]{CM17}
C.\-Y.\ Chang and Y.\ Mishiba, \textit{On multiple polylogarithms in characteristic $p$: $v$-adic vanishing versus $\infty$-adic Eulerianness}, Int. Math. Res. Not. IMRN (2019), no. 3, 923--947.

\bibitem[CPY19]{CPY14}
C.\-Y.\ Chang, M.\ A.\ Papanikolas and J.\ Yu, \textit{An effective criterion for Eulerian multizeta values in positive characteristic}, J. Eur. Math. Soc. (JEMS) \textbf{21} (2019), no. 2, 405-440.


\bibitem[CY07]{CY07} C.\-Y.\ Chang and J.\ Yu, \textit{Determination of algebraic relations among special zeta values in positive characteristic}, Adv. Math. \textbf{216} (2007), no. 1, 321-345.

\bibitem[Co82]{Co82} R.\ Coleman, \textit{Dilogarithms, regulators and p-adic L-functions}, Invent. Math. 69 (1982), no. 2, 171-208.

\bibitem[CGM19]{CGM19} C.\-Y.\ Chang, N.\ Green and Y.\ Mishiba, \textit{Taylor coefficients of $t$-motivic multiple zeta values and explicit formulae}, arXiv:1902.06879.



\bibitem[CGP10]{CGP10} B.\ Conrad, O.\ Gabber and G.\ Prasad, {\it Pseudo-reductive groups}, New Mathematical Monographs \textbf{17}, Cambridge University Press, Cambridge, 2010.

\bibitem[F04]{F04}
H.\ Furusho, \textit{$p$-adic multiple zeta values. I. $p$-adic multiple polylogarithms and the $p$-adic KZ equation}, Invent. Math. \textbf{155} (2004), no. 2, 253-286.

\bibitem[F06]{F06} H.\ Furusho, {\it Multiple zeta values and Grothendieck-Teichm\"uller groups}. Primes and knots, 49-82, Contemp. Math., \textbf{416}, Amer. Math. Soc., Providence, RI, 2006. 

\bibitem[F07]{F07} H.\ Furusho,  {\it $p$-adic multiple zeta values}. II. Tannakian interpretations. Amer. J. Math. \textbf{129} (2007), no. 4, 1105-1144.

\bibitem[FJ07]{FJ07} H.\ Furusho and A.\ Jafari, \textit{Regularization and generalized double shuffle relations for p-adic multiple zeta values},  Compositio Math. \textbf{143} (2007), 1089-1107.

\bibitem[G17]{G17} N.\ Green, \textit{Special zeta values coming from tensor powers of Drinfeld modules},  arXiv:1706.06048. 

\bibitem[Go79]{Go79}
D.\ Goss, \textit{$v$-adic zeta functions, $L$-series, and measures for function fields. With an addendum}, Invent. Math. \textbf{55} (1979), no. 2, 107-119.
\bibitem[Go96] {Go96} D.\ Goss, {\em Basic structures of function field
    arithmetic}, Springer-Verlag, Berlin, 1996.



\bibitem[HJ16]{HJ16}
U.\ Hartl and A.\-K.\ Juschka, \textit{Pink's theory of Hodge structures and the Hodge conjecture over function fields}, arXiv:1607.01412.


\bibitem[IKZ06]{IKZ06} K.\ Ihara, M.\ Kaneko and D.\ Zagier, \textit{Derivation and double shuffle relations for multiple zeta values}, Compositio Math. \textbf{142} (2006), 307-338.

\bibitem[M17]{M17} Y.\ Mishiba, \textit{On algebraic independence of certain multizeta values in characteristic $p$}, J. Number Theory \textbf{173} (2017), 512-528.  



\bibitem[P08]{P08} 
M.\ A.\ Papanikolas, Tannakian duality for Anderson-Drinfeld motives and algebraic independence of
Carlitz logarithms, Invent. Math. \textbf{171} (2008), no. 1, 123-174.


\bibitem[S97]{S97} S.\ K.\ Sinha, \textit{Periods of $t$-motives and transcendence}, Duke Math. J.  \textbf{88}, No. 3, 465-535. 

\bibitem[T92]{T92} D.\ S.\ Thakur, \textit{Drinfeld modules and arithmetic in the function fields}, Internat. Math. Res. Notices 1992, no. 9, 185-197. 

\bibitem [T04]{T04} D.\ S.\ Thakur, \textit{Function field arithmetic}, World Scientific Publishing, River Edge NJ, 2004.

\bibitem [T09]{T09} D.\ S.\ Thakur, \textit{Power sums with applications to multizeta and zeta zero distribution for $\FF_{q}[t]$}, Finite Fields Appl. \textbf{15} (2009), no. 4, 534-552. 

\bibitem[T09b]{T09b} D.\ S.\ Thakur, \textit{Relations Between Multizeta Values for $\FF_{q}[t]$}, Int. Math. Res. Not. IMRN (2009), no. 12,  2318-2346.

\bibitem[T10]{T10} D.\ S.\ Thakur, \textit{Shuffle relations for function field multizeta values}, Int. Math. Res. Not. IMRN (2010), no. 11, 1973-1980.

\bibitem[Wa08]{Wa08} M.\ Waldschmidt, \textit{Elliptic functions and transcendence}, Surveys in Number Theory, Dev. Math. \textbf{17} (2008), 143-188.

\bibitem[W89]{W89} G.\ W\"ustholz, \textit{Algebraische Punkte auf analytischen Untergruppen algebraischer Gruppen} (German) [Algebraic points on analytic subgroups of algebraic groups], 
Ann. of Math. (2) \textbf{129} (1989), no. 3, 501-517. 

\bibitem[Ya13]{Ya13} S.\ Yamamoto, \textit{Interpolation of multiple zeta and zeta-star values}, J. Algebra \textbf{385} (2013), 102-114. 


\bibitem [Yu91]{Yu91} J.\ Yu, \textit{Transcendence and special zeta values in characteristic $p$}, Ann. of Math. (2) \textbf{134} (1991), no.~1, 1--23.

\bibitem [Yu97]{Yu97} J.\ Yu, \textit{Analytic homomorphisms into Drinfeld modules}, Ann. of Math. (2) \textbf{145} (1997), no.~2, 215--233.

\bibitem [Zh16]{Zhao16} J.\ Zhao,  \textit{Multiple zeta functions, multiple polylogarithms and their special values},  Series on Number Theory and its Applications, 12. World Scientific Publishing Co. Pte. Ltd., Hackensack, NJ, 2016.


\end{thebibliography}

\end{document}